\documentclass[12pt]{article}
\usepackage{amsmath,amssymb,mathtools}
\usepackage{amsthm}
\usepackage{amscd}
\usepackage{epsfig}
\usepackage{graphics}
\usepackage{graphicx}
\usepackage{colordvi}
\usepackage{enumitem}
\usepackage{mathrsfs} 
\usepackage[colorlinks]{hyperref}
\usepackage{dsfont}
\usepackage{xcolor}
\usepackage{mparhack}
\usepackage{xcolor,lipsum}
\usepackage[displaymath, mathlines]{lineno}
\modulolinenumbers[5]
\nolinenumbers
\allowdisplaybreaks[4]
\usepackage[normalem]{ulem}
\renewcommand\sout{\bgroup\markoverwith{\textcolor{red}{\rule[0.5ex]{2pt}{1pt}}}\ULon}

\hypersetup{
linkcolor=blue!75!black,
urlcolor=blue!75!black,
citecolor=blue!75!black}

\theoremstyle{plain}

\theoremstyle{definition}

\theoremstyle{remark}

\numberwithin{equation}{section}

\newcommand{\1}{\mathds 1}
\newcommand{\error}{{\boldsymbol \gamma}}

\newcommand{\C}{{\mathscr C}}

\newcommand{\R}{{\Bbb R}}

\newcommand{\vep}{{\varepsilon}}

\newcommand{\lmt}{\longmapsto}

\newcommand{\sgn}{{\rm sgn}}

\renewcommand{\d}{{\rm d}}

\newcommand{\bs}{\boldsymbol}
\newcommand{\ms}{\mathscr}
\renewcommand{\P}{{\mathbb P}}

\newcommand{\E}{{\mathbb E}}

\newcommand{\spin}{{\bs \sigma}}
\newcommand{\sspin}{{\bs \rho}}
\newcommand{\spinn}{{\bs \vep}}
\renewcommand{\i}{{\mathbf i}}

\renewcommand{\k}{{\mathbf k}}
\newcommand{\la}{\langle}
\newcommand{\ra}{\rangle}
\renewcommand{\l}{{\ell}}
\newcommand{\D}{{\mathrm D}}

\newcommand{\vertiii}[1]{{\left\vert\kern-0.25ex\left\vert\kern-0.25ex\left\vert #1 
    \right\vert\kern-0.25ex\right\vert\kern-0.25ex\right\vert}}

\newtheoremstyle{slantthm}{10pt}{10pt}{\slshape}{}{\bfseries}{}{.5em}{\thmname{#1}\thmnumber{ #2}\thmnote{ (#3)}.}
\newtheoremstyle{slantrmk}{10pt}{10pt}{\rmfamily}{}{\bfseries}{}{.5em}{\thmname{#1}\thmnumber{ #2}\thmnote{ (#3)}.}

\oddsidemargin
-.445in \evensidemargin -.1in\marginparwidth 0.4in

\begin{document}

\newtheoremstyle{slantthm}{10pt}{10pt}{\slshape}{}{\bfseries}{}{.5em}{\thmname{#1}\thmnumber{ #2}\thmnote{ (#3)}.}
\newtheoremstyle{slantthmp}{10pt}{10pt}{\slshape}{}{\bfseries}{}{.5em}{\thmname{#1}\thmnumber{ #2}\thmnote{ (#3)}.}
\newtheoremstyle{slantrmk}{10pt}{10pt}{\rmfamily}{}{\bfseries}{}{.5em}{\thmname{#1}\thmnumber{ #2}\thmnote{ (#3)}.}

\theoremstyle{slantthm}
\newtheorem{thm}{Theorem}[section]
\newtheorem{prop}[thm]{Proposition}
\newtheorem{lem}[thm]{Lemma}
\newtheorem{cor}[thm]{Corollary}
\newtheorem{defi}[thm]{Definition}
\newtheorem{prob}[thm]{Problem}
\newtheorem{disc}[thm]{Discussion}
\newtheorem*{nota}{Notation}
\newtheorem*{conj}{Conjecture}
\theoremstyle{slantrmk}
\newtheorem{rmk}[thm]{Remark}
\newtheorem{eg}[thm]{Example}
\newtheorem{step}{Step}
\newtheorem{claim}{Claim}
\theoremstyle{plain}
\newtheorem{thmm}{Theorem}[section]

\title{{\bf Universality of Ghirlanda--Guerra identities and spin distributions in mixed $p$-spin models}}

\author{Yu-Ting Chen\footnote{Department of Mathematics, University of Tennessee, Knoxville, United States of America.}}

\maketitle

\abstract{We prove universality of the Ghirlanda--Guerra identities and spin distributions in the mixed $p$-spin models. The assumption for the universality of the identities requires exactly that the coupling constants have zero means and finite variances, and the result applies to the Sherrington--Kirkpatrick model. As an application, we obtain weakly convergent universality of spin distributions in the generic $p$-spin models under the condition of two matching moments. In particular, certain identities for 3-overlaps and 4-overlaps under the Gaussian disorder follow. Under the stronger mode of total variation convergence, we find that universality of spin distributions in the mixed $p$-spin models holds if mild dilution of connectivity by the Viana--Bray diluted spin glass Hamiltonians is present and the first three moments of coupling constants in the mixed $p$-spin Hamiltonians match. 
These universality results are in stark contrast to the characterization of spin distributions in the undiluted mixed $p$-spin models, which is known up to now that four matching moments are required in general. \\

\noindent\emph{Keywords:} Mixed $p$-spin models; the Ghirlanda--Guerra identities; Universality; Ultrametricity; the Viana--Bray diluted spin glass model\\

\noindent\emph{Mathematics Subject Classification (2000):} 82D30, 60K35, 82B44 
}

\tableofcontents

\section{Introduction}\label{sec:intro}
In this paper, we investigate disorder universality in the mixed $p$-spin models. This concerns the phenomenon that, in the thermodynamic limit, probability distributions of the models do not depend on the particular distribution of the coupling constants.
As a way to obtain universality of the Parisi ultrametricity, 
Auffinger and Chen \cite{AC_15} prove universality of the Gibbs measures in these models. The condition in \cite{AC_15} requires either enough
matching moments of the coupling constants (when compared to the corresponding moments of a standard Gaussian) or the absence of pure lower-spin Hamiltonians. 
Our interests here are mainly motivated by related beliefs in the applicability
of sharper conditions in the physics literature. We consider the Parisi ultrametricity and characterizations of the Gibbs measures under the condition of fewer matching moments,
for presumably the most important case where the pure $2$-spin Hamiltonians, namely the Sherrington--Kirkpatrick Hamiltonians, are present.

We work with the following mixed $p$-spin models throughout this paper unless otherwise mentioned. Given i.i.d. real-valued random variables $\xi_{\mathbf i}$, for $\mathbf i\in \{1,\cdots,N\}^p$ and $p\geq 2$, with zero means and unit variances, a mixed $p$-spin Hamiltonian is defined by 
\begin{align}\label{def:H}
H_\xi(\spin)=H_{N,\xi}(\spin)\equiv  \sum_{p\geq 2}H_{\xi,p}(\spin)+h\sum_{i=1}^N \sigma_i,
\end{align}
where $\spin=(\sigma_1,\cdots,\sigma_N)\in\Sigma_N\doteq \{-1,1\}^N$ and $h$ is a fixed real constant.
The Hamiltonian in (\ref{def:H}) is a sum of the external magnetic field $h\sum_{i=1}^N\sigma_i$ and the following pure $p$-spin Hamiltonians: 
\begin{align}\label{def:Hp}
H_{\xi,p}(\spin)=H_{N,\xi,p}(\spin)\equiv \frac{\beta_p}{N^{(p-1)/2}}\sum_{\mathbf i\in \{1,\cdots, N\}^p}\xi_{\mathbf i}\sigma_{\mathbf i},
\end{align}
with the notation 
\[
\sigma_{\mathbf i}=\sigma_{i_1}\cdots \sigma_{i_p}
\]
for $\mathbf i=(i_1,\cdots,i_p)$.
Here in (\ref{def:Hp})  and throughout this paper, the real constants $\beta_p$ are temperature parameters assumed to satisfy $\sum_{p=2}^\infty \beta_p^2 2^p<\infty$ so that the Hamiltonian $H_\xi$ is well-defined almost surely.  
The above assumption on the disorder $\xi=(\xi_{\mathbf i})$ means that the random field $H_\xi(\spin)$, $\spin\in \Sigma_N$, is centered and has the same covariance matrix, with entries typically of order $N$, as the covariance matrix under the (standard) Gaussian disorder. Nonetheless, useful symmetry properties which are valid under the Gaussian disorder can be broken.   
The Gibbs measure on $\Sigma_N$ induced by the Hamiltonian in (\ref{def:H}) is denoted by
\begin{align}\label{Gibbs}
G_{N,\xi}(\{\spin\})\equiv \frac{e^{H_{N,\xi}(\spin)}}{Z(H_{N,\xi})},\quad \spin \in \Sigma_N,
\end{align}
where the partition function $Z(H_{N,\xi})$ enters the definition of $G_{N,\xi}$ as a normalizing factor. 

Mathematical investigations of disorder universality in the mixed $p$-spin models 
have been focused around the Sherrington--Kirkpatrick model~\cite{Sherrington_1975} and go back  to Talagrand~\cite{Talagrand_02} where a comparison principle for expected  extrema of correlated random variables is studied. For the Sherrington--Kirkpatrick model, \cite{Talagrand_02} proves that expected values of
the rescaled ground-state energy $N^{-1}\min_{\spin\in \Sigma_N} H_{N,\,\cdot\, ,2}(\spin)$ under the Gaussian disorder and the Bernoulli disorder coincide
in the thermodynamic limit. The method there  
considers the analogous universality question
for the free-energy densities, 
\[
\frac{1}{N}\log Z(H_{N,\,\cdot\,,2});
\]
working with these quantities has the advantage that, by a generalized Gaussian integration by parts, the rates of change of these free-energy densities along an interpolation between the two disorders vanish in the limit of large system size.
Later on, Guerra and Toninelli~\cite[Section~4.2]{Guerra_2002} show that their method for the existence of limits of free-energy densities in the Sherrington--Kirkpatrick model extends to those under general disorder. Precisely, the assumption in \cite{Guerra_2003} states symmetric distributions with finite fourth moments of coupling constants. Carmona and Hu~\cite{CH_06} then extend Talagrand's method in \cite{Talagrand_02},
and obtain a mathematical proof that limiting free-energy densities in the Sherrington--Kirkpatrick model depend only on the first two moments of coupling constants (cf. Kirkpatrick and Sherrington \cite{Kirkpatrick_1978}), as well as an analogous result for the rescaled ground-state energy. (Generalizations to the mixed $p$-spin models are also obtained in \cite{CH_06}.) Regarding the additional technical assumption in \cite{CH_06} of finite third moments, Chatterjee~\cite[Section~3]{C_05} observed that it can be circumvented by Lindeberg's argument for the central limit theorem (cf. \cite{Trotter_1959}). 

The first main result of this paper  (Theorem~\ref{thm:1}) proves disorder universality of the Ghirlanda--Guerra identities in the mixed $p$-spin models defined by (\ref{def:H}). Therefore the identities under the Gaussian disorder remain valid when coupling constants satisfying the condition of two matching moments are applied. Since the discovery of these identities in \cite{Ghirlanda_1998}, they have proven fundamental to understand the main characteristics of the mixed $p$-spin models regarding the overlap arrays:
\begin{align}\label{def:overlap}
R_{\l,\l'}= \frac{1}{N}\sum_{i=1}^N\sigma^\l_{i}\sigma^{\l'}_{i}, \quad \ell\neq \ell',
\end{align}
where $\spin^1,\spin^2,\cdots $ are replicas as i.i.d. samples from the Gibbs measures over $N$ sites.
In particular, the identities 
are the driving force for the mathematical proofs of the Parisi formula and the Parisi ultrametricity in the hands of Panchenko \cite{P_13, Panchenko_2014_Parisi}, where the work \cite{Panchenko_2014_Parisi} extends Talagrand's ground-breaking proof for the Parisi formula in the even $p$-spin models \cite{Talagrand_2006_P}. The major step of the simple proof of Theorem~\ref{thm:1} shows that a generalization of the Gaussian integration by parts as in \cite{CH_06} suffices. It is used to obtain certain integration by parts  involving the internal energy densities $N^{-1}H_{N,\xi}$, 
which is central to the derivation of the identities. As immediate consequences of the universality of the identities, we obtain, 
in the particular case of generic $p$-spin models (see (\ref{def:generic}) for the definition), 
universality of the Gibbs measures of pure states \cite{Mezard_1984} in the sense of Talagrand's construction \cite[Theorem~2.4]{Talagrand_2009} and universality of the Parisi ultrametricity by Panchenko's theorem \cite[Theorem~1]{P_13}.

It is proven in \cite{AC_15}, among other things, universality of the Parisi ultrametricity in the mixed $p$-spin models. The approach in \cite{AC_15} proceeds with universality of the Gibbs measures in the sense that   the Gibbs measures over finitely many configurations  are compared under different disorders by total variation distances. Precisely, the following convergences are obtained in \cite{AC_15}:
\begin{align}\label{AC_result}
\lim_{N\to\infty}\sup_{F:\Sigma^n_N\to [-1,1]}\left|\E\langle F\rangle_{H_{N,\xi}}-\E\langle F\rangle_{H_{N,g}}\right|=0,\quad \forall\; n\in \Bbb N,
\end{align}
where
\[
\langle F\rangle_{H_{N,\xi}} =\langle F(\spin^1,\cdots,\spin^n)\rangle_{H_{N,\xi}}=\sum_{\spin^1,\cdots,\spin^n\in \Sigma_N}F(\spin^1,\cdots,\spin^n)G_{N,\xi}(\{\spin^1\})\cdots G_{N,\xi}(\{\spin^n\}).
\]
The methods in \cite{CH_06, C_05} by interpolations and generalized integration by parts discussed above are applied for the proof of (\ref{AC_result}) in \cite{AC_15}. See also \cite{Crawford_2007, Genovese_2012} for similar applications to other spin glass models.

In stark contrast to the various universality phenomena discussed before (\ref{AC_result}),  
at least one of the following conditions
is required in \cite[Theorem~4.3]{AC_15} to validate (\ref{AC_result}): (1) the first four  moments of the coupling constants match those of a standard Gaussian, or (2) there is absence of both pure $2$-spin Hamiltonians and pure $3$-spin Hamiltonians. 
While Theorem~\ref{thm:1} improves conditions for universality of the Parisi ultrametricity, 
\cite{AC_15} still leaves open a natural 
question whether Gibbs measures in the Sherrington--Kirkpatrick model can be characterized by fewer matching moments of the coupling constants. 
At least it should raise the question whether in the limit, disorder as simple as the Bernoulli disorder, which is of interest in the study of the Edward-Anderson model (e.g. \cite{Toulouse_1986, Nishimori_2001}),
leads to the same probabilistic features of the mean-field spin glass models in terms of
Gibbs measures.

We revisit universality of Gibbs measures in the mixed $p$-spin models in the next two main results of this paper. The second main result studies spin distributions, namely finite-dimensional marginals of coordinates of replicas, and  
obtains their weakly convergent universality
in the particular case of generic $p$-spin models (Theorem~\ref{thm:2}). The third main result continues to use the total variation distances as in  (\ref{AC_result}) to compare Gibbs measures 
  in the mixed $p$-spin models under two different disorders, and now independent Viana--Bray diluted $2$-spin glass Hamiltonians \cite{Viana_1985} (defined in (\ref{def:H'})) with small magnitude are present (Theorem~\ref{thm:3}). The rich
structural similarities in the Gibbs measures under different disorders implied by the universality of the Ghirlanda--Guerra identities 
allows us to work beyond the scope of earlier methods for universality of spin glass models, and thereby, make possible these two main results.

The universality of spin distributions in the generic $p$-spin models holds under precisely the condition of two matching moments. The major step of the proof considers an insightful application of the cavity method  in \cite{Panchenko_2013_spin}, which shows nontrivial connections between the spin distributions and joint distributions of the overlap arrays, and we obtain its extension under general disorder. 
(See also \cite[Proposition~A.7]{AJ}.) 
The proof shows that the cavity equation considered in \cite{Panchenko_2013_spin} (cf. (\ref{eq:cavityg})), which features Gaussian factors from the Hamiltonians over cavity sites, is approximately valid under general disorder. This approximate equation allows the key application as in \cite{Panchenko_2013_spin} that the spin distributions depend almost continuously on joint distributions of the overlap arrays defined in (\ref{def:overlap}). With this connection, the uniqueness and universality of the joint distributions under the condition of two matching moments, which are readily satisfied by the generic $p$-spin models, extend to the spin distributions. 
In particular, 
the universality of spin distributions leads to certain identities for $3$-overlaps and $4$-overlaps under the Gaussian disorder whenever pure $2$-spin Hamiltonians are present (Proposition~\ref{prop:1}). These identities seem to be new to the 
generic $p$-spin models although they
already appear in the literature of diluted spin glass models.

The third main result considers removal of the abundance of higher-spin Hamiltonians defining the generic $p$-spin models and applies to any choice of temperature parameters (Theorem~\ref{thm:3}). 
We require three matching moments and existence of fourth moments in the mixed $p$-spin Hamiltonians. The proof shows that when the interpolation method between two disorders as in \cite{CH_06,AC_15} is applied, mild dilution of connectivity in the mixed $p$-spin physical systems by the Viana--Bray Hamiltonians, as well as the Ghirlanda--Guerra identities, lead to vanishing rates of change of the Gibbs expectations. 
The way we dilute the mixed $p$-spin Hamiltonians is chosen so that
the connectivity of the perturbative systems can range from order $1$ to  any growing order obeying the limiting behavior $o(N)$ and infinite temperature limit is taken. In particular, limiting free-energy densities in 
the diluted mixed $p$-spin models and the undiluted ones coincide. 

It is arguable whether 
any of the diluting perturbations may change the nature of Gibbs measures in the mixed $p$-spin models. 
The last main result, however, shows that, even in the absence of the regular overlap structure in the generic $p$-spin models, the condition of four matching moments for the universality of Gibbs measures in the mixed $p$-spin models is not robust, and can be improved if there are seemingly natural impurities as the diluting spin glass Hamiltonians.

\paragraph{\bf Organization of the paper.} Section~\ref{sec:UGG} presents the first universality result in this paper, which is for the Ghirlanda--Guerra identities. Our study of the universality of spin distributions is divided into Sections~\ref{sec:UMO} and \ref{sec:convtv}. The former section considers the universality under weak convergence, whereas the latter section considers the universality under total variation convergence. Between these two sections, we present in Section~\ref{sec:D4} some combinatorial connections between Gibbs measures and multi-overlaps. These connections refine some results in the earlier sections to prepare for
Section~\ref{sec:convtv}.
For completeness, we close this paper with a generalization of the Gaussian integration by parts  in Section~\ref{sec:AIBP}.  
\vspace{-.2cm}

\paragraph{\bf Acknowledgements.} Partial supports from the Center of Mathematical Sciences and Applications at Harvard University during the author's previous position and from NCTS Taipei during a visit
are gratefully acknowledged. The author is indebted to Prof. Horng-Tzer Yau for enlightening discussions and would like to thank Philippe Sosoe for numerous clarifying conversations. The author also thanks an anonymous referee for very helpful comments that improve upon the presentation of this paper.

\section{Universality of the Ghirlanda--Guerra identities}\label{sec:UGG}
In this section, we prove universality of the Ghirlanda--Guerra identities. 
First, let us introduce some notations to facilitate the use of the Gaussian integration by parts and its generalizations throughout this paper. 
For any $p\geq 2$ and $\mathbf i\in \{1,2,\cdots, N\}^p$, we define
a multiplication operator $\D_{\mathbf i}$, acting on the set of real-valued functions of finitely many spin configurations over $N$ sites, by
 \begin{align}\label{def:Dij}
 \D_{\mathbf i}F(\spin^1,\cdots,\spin^n,\spin^{n+1})\doteq 
\left(\sum_{\ell=1}^n\sigma^\ell_{\mathbf i}-n\sigma^{n+1}_{\mathbf i}\right)\cdot  F(\spin^1,\cdots,\spin^n).
 \end{align} 
Also, for a function $f=f(x)$, we write $\partial_x^jf$ for its $j$-th derivative.

The following theorem is the first main result of this paper.

\begin{thm}
\label{thm:1}
Consider the mixed $p$-spin model defined by the Hamiltonian in (\ref{def:H}). Assume that the coupling constants $(\xi_{\mathbf  i})$ are i.i.d. and satisfy
\begin{align}\label{assump:mom}
\E\xi_{\mathbf i}=0\quad\mbox{and}\quad \E\xi_{\mathbf i}^2\in (0,\infty).
\end{align}
Then for all $p\geq 2$ and sequences of temperature parameters $(\beta_q)_{q\geq 2}$ such that $\sum_{q=2}^\infty 2^q\beta_q^2<\infty$ and $\beta_p\neq 0$, it holds that, for every $n\in \Bbb N$,
\begin{align}\label{GG:lim}
\lim_{N\to\infty}\sup_{F:\Sigma^n_N\to [-1,1]}\left|\E\langle R_{1,n+1}^pF\rangle_{H_{N,\xi}}-\frac{1}{n}\E\langle R^p_{1,2}\rangle_{H_{N,\xi}} \E\langle F\rangle_{H_{N,\xi}} -\frac{1}{n}\sum_{\ell=2}^n \E\langle R^p_{1,\ell}F\rangle_{H_{N,\xi}}\right|=0.
\end{align}
\end{thm}
\begin{proof}[\bf Proof.]
The proof uses the self-averaging of internal energy densities as in the case of the Gaussian disorder (cf. \cite[Theorem~3.8]{Panchenko_2013}), and now we need to validate the analogous property under the disorder $\xi$. We suppress subscripts `$N$' and `$\xi$' in Hamiltonians   
whenever the context is clear 
and write $\langle \,\cdot\,\rangle=\langle\, \cdot\,\rangle_{H_{N,\xi}}$ throughout this proof. We may assume  $\E\xi_{\mathbf i}^2=1$ by rescaling the temperature parameters. \\

\noindent \hypertarget{step1-1}{{\bf Step 1.}}
 We show that  for any function $F=F(\spin^1,\cdots,\spin^n):\Sigma^n_N\to [-1,1]$,
the difference
\begin{align}\label{diffF}
\E\left\langle \frac{H_p(\spin^1)}{N}F\right\rangle-\E\left\langle \frac{H_p}{N}\right\rangle \E\left\langle F\right\rangle
\end{align}
approximates 
\begin{align}\label{diffF02}
\beta_p^2 \left(\sum_{\ell=2}^n \E\langle R^p_{1,\ell}F\rangle-n\E\langle R_{1,n+1}^pF\rangle+\E\langle R^p_{1,2}\rangle \E\langle F\rangle \right)
\end{align}
in the limit $N\to\infty$. This difference is exactly zero under the Gaussian disorder by the Gaussian integration by parts.

We start with the following identities:
\begin{align}\label{Gibbs:der}
\partial_{x}^j\langle F\rangle=\frac{\beta_p^j}{N^{j(p-1)/2}}\langle \D_{\mathbf i}^jF\rangle_{\xi_\i=x},\quad j\in \Bbb N,
\end{align}
where we use the definition (\ref{def:Dij}) of $\D_{\mathbf i}$ and $\langle \,\cdot\, \rangle_{\xi_{\mathbf i}=x}$ is the Gibbs expectation defined by setting $\xi_{\mathbf i}$ in $\langle \,\cdot\,\rangle$ to be $x$.
Then with 
$\error^2_{\xi_{\mathbf i}}$ defined by (\ref{gamma2}) and the notation
$f_{\mathbf i}(x)\equiv \E\langle \sigma^1_{\mathbf i}F\rangle_{\xi_\i=x}$,
we obtain from Proposition~\ref{prop:AIBP} for (\ref{HIBP1}) below that
\begin{align}
\E\left\langle \frac{H_p(\spin^1)}{N}F\right\rangle
=&\frac{\beta_p}{N^{(p-1)/2+1}}\sum_{\mathbf i\in \{1,\cdots,N\}^p}\E\xi_{\mathbf i} f_{\mathbf i}(\xi_{\mathbf i})\notag\\
=&\,\frac{\beta_p}{N^{(p-1)/2+1}}\sum_{\mathbf i\in \{1,\cdots,N\}^p}\E\partial_x f_{\mathbf i}(\xi_{\mathbf i})+\frac{\beta_p}{N^{(p-1)/2+1}}\sum_{\mathbf i\in \{1,\cdots,N\}^p}\error^2_{\xi_{\mathbf i}}\big(\partial^2_xf_{\mathbf i}\big)\label{HIBP1}\\
=&\,\beta_p^2\E\left\langle \left(\sum_{\ell=1}^n R^p_{1,\ell}-nR^p_{1,n+1}\right)F\right\rangle+\frac{\beta_p}{N^{(p-1)/2+1}}\sum_{\mathbf i\in \{1,\cdots,N\}^p}\error^2_{\xi_{\mathbf i}}\big(\partial_x^2 f_{\mathbf i}\big),\label{HIBP2}
\end{align}
where the last equality follows from (\ref{Gibbs:der}) with $j=1$ and the definition (\ref{def:Dij}) of $\D_{\mathbf i}$. 
For the last term in (\ref{HIBP2}), we use the explicit form (\ref{gamma2}) of $\error^2_{\xi_{\mathbf i} }$ and (\ref{IBP2-1})  to get
\begin{align}
\begin{split}
&\varlimsup_{N\to\infty}\sup_{F:\Sigma^n_N\to [-1,1]}\left|\frac{\beta_p}{N^{(p-1)/2+1}}\sum_{\mathbf i\in \{1,\cdots,N\}^p}\error^2_{\xi_{\mathbf i}}\big(\partial_x^2f_{\mathbf i}\big)\right|\\
\leq &\lim_{N\to\infty}\frac{|\beta_p|}{N^{(p-1)/2+1}}\sum_{\mathbf i\in \{1,\cdots,N\}^p}\left(\E \left[|\xi_{\mathbf i}|\int_{0}^{|\xi_{\mathbf i}|}\min\big\{2\|\partial_x f_{\mathbf i}\|_\infty,\|\partial_x^2f_{\mathbf i}\|_\infty x\big\}\d x \right]
+\frac{\E\xi_{\mathbf i}^2 \|\partial_x^3f_{\mathbf i}\|_\infty}{2}\right)=0.
\label{diffF2}
\end{split}
\end{align} 
Here, the last equality follows from (\ref{Gibbs:der}) and the assumption that the coupling constants are identically distributed and satisfy $\E \xi^2_\i<\infty$ so that for any $\vep>0$, all the summands can be bounded by $\vep/N^{(p-1)/2}$ whenever $N\geq N_0$ for some $N_0$ depending on $n$ and $\beta_p$.  
The foregoing equality and (\ref{HIBP2}) imply that
\begin{align}
\lim_{N\to\infty}\sup_{F:\Sigma^n_N\to [-1,1]}\left|\E\left\langle \frac{H_p(\spin^1)}{N}F\right\rangle
-\beta_p^2\E\left\langle \left(\sum_{\ell=1}^n R^p_{1,\ell}-nR^p_{1,n+1}\right)F\right\rangle\right|=0.\label{diffF3}
\end{align}
The required approximation of (\ref{diffF}) by (\ref{diffF02}) under the same mode of uniform convergence in (\ref{diffF3}) then follows. We remark that if $\E|\xi_{\mathbf i}|^3<\infty$, (\ref{Gibbs:der}) and (\ref{diffF2}) show that the rate of convergence for (\ref{diffF3}) is of the order $\mathcal O(1/N^{(p-1)/2})$.\\

\noindent \hypertarget{step1-2}{{\bf Step 2.}}
To obtain (\ref{GG:lim}), it remains to show that the differences in (\ref{diffF}) tend to zero uniformly in $F:\Sigma^n_{N}\to [-1,1]$, and so a proof of the following equality suffices:
\begin{align}\label{Hlim}
\lim_{N\to \infty}\frac{1}{N}\E\langle|H_p-\E\langle H_p\rangle | \rangle=0.
\end{align}
In the rest of this proof, $Z_N(x)$ denotes the partition function defined by $H_{N,\xi}$ with $\beta_p=x$. 

To get (\ref{Hlim}),
we need the following three properties of $\lim_{N\to\infty}N^{-1}\log Z_N(x)$.
First, we recall that under the moment assumption (\ref{assump:mom}), the limit in probability of the free-energy densities $N^{-1}\log Z_N(x)$ is the same as the limit of the expectations of their counterparts under the Gaussian disorder (\cite[Lemma~8]{CH_06} and \cite[Section~3]{C_05}), which is given by the Parisi formula \cite[Theorem~3.1]{Panchenko_2013}.
Second, a standard martingale difference argument by Burkholder's inequality \cite[Theorem~3.2]{Burkholder_1973} as in the proof of \cite[Lemma~8]{CH_06} 
shows that for each $x\in \R$, 
\begin{align}
\begin{split}\label{Burk}
\E\left[\left(\frac{1}{N}\log Z_N(x)-\frac{1}{N}\E\log Z_N(x)\right)^2\right]\leq &\sum_{q=2,q\neq p}^\infty \frac{\beta_q^2}{N^{q+1}}\sum_{\mathbf i\in \{1,\cdots,N\}^q}\E\big[\big(|\xi_{\mathbf i}|+\E|\xi_{\mathbf i}|\big)^2\big]\\
&+\frac{x^2}{N^{p+1}}\sum_{\mathbf i\in \{1,\cdots,N\}^p}\E\big[\big(|\xi_{\mathbf i}|+\E|\xi_{\mathbf i}|\big)^2\big].
\end{split}
\end{align}
For fixed $x\in \R$, the foregoing inequality and the convergence of $N^{-1}\E\log Z_N(x)$ imply that the sequence $\{N^{-1}\log Z_N(x)\}$  
converges in $L_2(\P)$ to the Parisi formula. Third, the limiting free-energy density is everywhere differentiable in $x$ since the corresponding Parisi formula is \cite[Theorem~3.7]{Panchenko_2013}.
These three properties are enough to apply \cite[Theorem~3.8]{Panchenko_2013} to validate (\ref{Hlim}). The proof of Theorem~\ref{thm:1} is complete.
\end{proof}

The following corollary will be used in Section~\ref{sec:convtv}.

\begin{cor}\label{cor:1}
Under the conditions of Theorem~\ref{thm:1} for fixed $p\geq 2$, it holds that
\begin{align}
\begin{split}\label{eq:D2}
&\lim_{N\to\infty}\sup_{F:\Sigma^n_N\to [-1,1]}\Bigg|2\sum_{1\leq \ell<\ell'\leq n}\E\langle R^p_{\ell,\ell'}F\rangle_{H_{N,\xi}}-2n\sum_{\ell=1}^n\E\langle R^{p}_{\ell,n+1}F\rangle_{H_{N,\xi}} \\
&\hspace{6cm}+n(n+1)\E\langle R^{p}_{n+1,n+2}F\rangle_{H_{N,\xi}}\Bigg|=0,\quad \forall\;n\in \Bbb N.
\end{split}
\end{align}
\end{cor}
\begin{proof}[\bf Proof.]
The proof is identical to some argument in the proof of \cite[Theorem~2]{CG_stability}. We give the details for the convenience of the reader.  For $1\leq \ell\leq n$, set
\begin{align}\label{def:DeltaF}
\Delta_\ell F(\spin^1,\cdots,\spin^n,\spin^{n+1})\doteq 
F(\spin^1,\cdots,\spin^n)-F(\spin^1,\cdots,\spin^{\ell-1},\spin^{\ell+1},\cdots,\spin^{n+1}) ,
\end{align}
and so by (\ref{GG:lim}) the following tends to zero uniformly in $F:\Sigma^n_N\to [-1,1]$: 
\begin{align}
&(n+1)\E\langle R^p_{\ell,n+2}\Delta_\ell F\rangle_{H_{N,\xi}}-\sum_{\stackrel{\scriptstyle 1\leq \ell'\leq n+1}{\ell'\neq \ell}}\E\langle R^p_{\ell,\ell'}\Delta_\ell F\rangle_{H_{N,\xi}}\notag\\
=&n\E\langle R^p_{\ell,n+1}F\rangle_{H_{N,\xi}}-(n+1)\E\langle R^p_{n+1,n+2}F\rangle_{H_{N,\xi}}-\sum_{\stackrel{\scriptstyle 1\leq \ell'\leq n}{\ell'\neq \ell}}\E\langle R^p_{\ell,\ell'}F\rangle+\sum_{\ell'=1}^n\E\langle R^p_{\ell',n+1}F\rangle_{H_{N,\xi}}.\label{eq:D2-2}
\end{align} 
We deduce (\ref{eq:D2}) upon
summing the right-hand sides of (\ref{eq:D2-2}) over $\ell\in \{1,\cdots,n\}$.  
\end{proof}

By the use of the functions $\Delta_\ell F$, (\ref{eq:D2}) depends only 
the thermal part of the Ghirlanda--Guerra identities. Indeed, an inspection 
of Step~\hyperlink{step1-1}{1} in the proof of Theorem~\ref{thm:1} shows that (\ref{eq:D2})
only needs the self-averaging of the thermal fluctuations of Gibbs measures as follows:
\begin{align}\label{internal-diff}
\lim_{N\to\infty}\sup_{F:\Sigma^n_N\to [-1,1]}\left|\E\left\langle \frac{H_p(\spin^1)}{N}F\right\rangle_{H_{N,\xi}}-\E\left\langle \frac{H_p}{N}\right\rangle_{H_{N,\xi}} \langle F\rangle_{H_{N,\xi}}\right|=0.
\end{align}

\section{Weak convergence of spin distributions}\label{sec:UMO}
Our main goal in this section is to characterize spin distributions in the mixed $p$-spin models by the universality of the Ghirlanda--Guerra identities in Theorem~\ref{thm:1}. Weak convergence of distributions is considered.

At first, we consider the case of generic $p$-spin models. In these models, the temperature parameters $(\beta_p)_{p\geq 2}$ satisfy the following Stone-Weierstrass-type condition:
\begin{align}\label{def:generic}
\mbox{${\rm span}\{x\mapsto x^p;\beta_p\neq 0\}$ is dense in $\C([-1,1])$}.
\end{align}
Here, $\C([-1,1])$ is the space of real-valued continuous functions on $[-1,1]$ equipped with the supremum norm. 

\begin{thm}
\label{thm:2}
Suppose that (\ref{assump:mom}) holds with $\E\xi_{\mathbf i}^2=1$ and a sequence of temperature parameters satisfying (\ref{def:generic}) is in force. Let $\Pi_\xi$ be the weak limit  of the sequence of spin distributions
\[
\big((\sigma^\ell_i)_{\ell,i\in \Bbb N}, \E\langle\, \cdot\,\rangle_{H_{N,\xi}}\big)_{N\geq 1}, 
\]
where $\sigma^\ell_i=0$ for all $i>N$ under $\E\langle\, \cdot\,\rangle_{H_{N,\xi}}$.
Then 
$\Pi_\xi=\Pi_g$, where $g$ denotes the standard Gaussian disorder. 
\end{thm}
\begin{proof}[\bf Proof.]
We show universality of joint distributions of the overlap arrays $(R_{\ell,\ell'})_{\ell\neq \ell'}$ in Step~\hyperlink{step2-1}{1}, and then prove in the rest of the proof that this particular universality implies the required universality of spin distributions.\\

\noindent \hypertarget{step2-1}{{\bf Step 1.}}
The limiting free energy density under $\xi$ is given by the Parisi formula by its disorder universality \cite{CH_06,C_05}, 
Also, by the explicit formulas for partial derivatives of the Parisi formula with respect to temperature parameters \cite[Theorem~3.7]{Panchenko_2013}, the $p$-th annealed moments of $R_{1,2}$ under $\xi$ and $g$ coincide in the limit  whenever $\beta_p\neq 0$. Hence, the annealed distributions of $R_{1,2}$ under the two disorders in the limit are equal by (\ref{def:generic}). By the  Baffioni-Rosati theorem (cf. \cite[Theorem~2.13]{Panchenko_2013} or \cite[Section~15.3]{Talagrand_v2}) applied to any weak subsequential limit of the sequence of annealed distributions of $(R_{\ell,\ell'})_{\ell\neq \ell'}$ under $\xi$, we obtain the weak convergence of this sequence and the limit $\Pi_\xi\big((R_{\ell,\ell'})_{\ell\neq \ell'}\in \,\cdot\,\big)$ is given by $\Pi_g\big((R_{\ell,\ell'})_{\ell\neq \ell'}\in \,\cdot\,\big)$.\\

\noindent \hypertarget{step2-2}{{\bf Step 2.}} We set up some notations for applications of the cavity method. Fix $n'\geq 1$. We apply the cavity method to the system over $N+n'$ sites where the first $n'$ coordinates play the role of cavity coordinates, and write $\sspin=(\spinn,\spin)\in \Sigma_{n'}\times \Sigma_N$. We decompose the Hamiltonian $H_{N+n'}(\sspin)$ defined by 
(\ref{def:H}) according to the number of occurrences of sites $1,\cdots,n'$ in $\mathbf i\in \{1,\cdots,N+n'\}^p$ and get
\begin{align}\label{HN:decg}
H_{N+n',\xi}(\sspin)\stackrel{\rm (d)}{=}H'_{N,\xi}(\spin)+\sum_{j=1}^{n'}\vep_j h_{N,j,\xi}(\spin)+r_{N,\xi}(\spinn,\spin)+h\sum_{i=1}^{{N+n'}}\rho_i,
\end{align}
for $\sspin=(\spinn,\spin)\in \Sigma_{n'}\times \Sigma_N$,
where the $q$-th term on the right-hand side, for $q=1,2$, has $q-1$ numbers of occurrences of spins in the first $n'$ coordinates of $\sspin$, and the third term has $\geq 2$ numbers of occurrences of these spins.
 In detail, we have 
\begin{align}
H'_{N,\xi}(\spin)=&\sum_{p\geq 2}\frac{\beta_p}{(N+n')^{(p-1)/2}}\sum_{\mathbf k\in \{1,\cdots,N\}^p}\xi_{\mathbf k}\sigma_{\mathbf k},\label{HN:dec1H'Ng}\\
h_{N,j,\xi}(\spin)=&\sum_{p\geq 2}\frac{\beta_p}{(N+n')^{(p-1)/2}}\sum_{a=1}^{{p\choose 1}} \sum_{\mathbf k\in \{1,\cdots,N\}^{p-1}}\xi_{j,a,\mathbf k}\sigma_{\mathbf k},\quad 1\leq j\leq n',\label{HN:dec1hNg}\\
\begin{split}
r_{N,\xi}(\sspin)=&r_{N,\xi}(\spinn,\spin)
=\sum_{p\geq 2}\frac{\beta_p}{(N+n')^{(p-1)/2}}\sum_{\mathbf i\in \ms I(p)}\xi_{\mathbf i}\rho_{\mathbf i},\label{HN:dec1rNg}
\end{split}
\end{align}
where the couplings constants on the right-hand sides are i.i.d. with the same laws as $\xi_{\mathbf i}$ and $\ms I(p)\subset \{1,\cdots,N+n'\}^p$ satisfies
$|\ms I(p)|=\sum_{\ell=2}^p{p\choose \ell}(n')^\ell N^{p-\ell}$.

For the following argument,
we also introduce an independent set of standard Gaussian coupling constants and define analogues of $h_{N,j,\xi}$ as
\begin{align}
h_{N,j,g}(\spin)=&\sum_{p\geq 2}\frac{\beta_p}{(N+n')^{(p-1)/2}}\sum_{a=1}^{{p\choose 1}} \sum_{\mathbf k\in \{1,\cdots,N\}^{p-1}}g_{j,a,\mathbf k}\sigma_{\mathbf k},\quad 1\leq j\leq n'.\label{HN:dec1hN}
\end{align}

\noindent \hypertarget{step2-3}{{\bf Step 3.}} 
In this step, we show that the Gibbs measure defined by $H_{N+n',\xi}$ is not changed in the limit of $N\to\infty$ if we replace $h_{N,j,\xi}$ by $h_{N,j,g}$ and remove $r_{N,\xi}$.  
Specifically, setting 
\begin{align}\label{HN:decg2}
H^t_{N+n',\xi,g}(\sspin)=H'_{N,\xi}(\spin)+\sum_{j=1}^{n'}\vep_j\big(\sqrt{t}h_{N,j,\xi}(\spin)+\sqrt{1-t}h_{N,j,g}(\spin)\big)+t r_{N,\xi}(\sspin)+h\sum_{i=1}^{{N+n'}}\rho_i
\end{align}
for $0\leq t\leq 1$ so that $H^1_{N+n',\xi,g}=H_{N+n',\xi}$, we show that
\begin{align}\label{eq:decg3}
\lim_{N\to\infty}\sup_{F:\Sigma^n_N\to [-1,1]}\big|\E\langle F\rangle_{1}-\E\langle F\rangle_{0}\big|=0,\quad \forall\;n\in \Bbb N,
\end{align}
where $\langle \,\cdot\,\rangle_t$ denotes the Gibbs expectation under $H^t_{N+n',\xi,g}$. 

For any $F=F(\sspin^1,\cdots,\sspin^n):\Sigma^{n}_{N+n'}\to [-1,1]$, we have
\begin{align}
&\E\langle F\rangle_1-\E\langle F\rangle_0=\int_0^1 \E\frac{\d}{\d t}\langle F\rangle_t \d t\notag\\
\begin{split}
=&\int_0^1 \frac{1}{2\sqrt{t}}\sum_{j=1}^{n'}\E\left\langle \left(\sum_{\ell=1}^{n}\vep^\ell_jh_{N,j,\xi}(\spin^\ell)-n\vep^{n+1}_jh_{N,j,\xi}(\spin^{n+1})\right)F \right\rangle_t\d t\\
&-\int_0^1 \frac{1}{2\sqrt{1-t}}\sum_{j=1}^{n'}\E\left\langle \left(\sum_{\ell=1}^n\vep^\ell_jh_{N,j,g}(\spin^\ell)-n\vep^{n+1}_jh_{N,j,g}(\spin^{n+1})\right)F \right\rangle_t\d t\\
&+\int_0^1 \E\left\langle \left(\sum_{\ell=1}^{n}r_{N,\xi}(\sspin^\ell)-nr_{N,\xi}(\sspin^{n+1})\right)F \right\rangle_t\d t.
\label{hxig-diff}
\end{split}
\end{align}
Notice that the two processes $h_{N,j,\xi}$ and $h_{N,j,g}$ have the same covariance function by the assumption of two matching moments of $\xi$. Hence, by Proposition~\ref{prop:AIBP} and the Gaussian integration by parts, the sum of the first and second integrals in (\ref{hxig-diff}) can be bounded by $n'\cdot n\cdot \sup_{j,\ell,t}|\mathcal E_{j,\ell}(t)|$. Here, $\mathcal E_{j,\ell}(t)$ is the remainder term from the application of (\ref{IBP2}) to $\E\big\langle 
\vep^\ell_j h_{N,j,\xi}(\spin^\ell)F \big\rangle_t
$. By (\ref{Gibbs:der}),
\begin{align*}
|\mathcal E_{j,\ell}(t)|\leq C\sum_{p\geq 2}\sum_{a=1}^{{p\choose 1}}\sum_{\mathbf k\in \{1,\cdots,N\}^{p-1}}\frac{\beta_p^2}{N^{p-1}}\left(\E|\xi_{j,a,\mathbf k}|\int_{0}^{|\xi_{j,a,\mathbf k}|}\min\left\{1,\frac{|\beta_p |x}{N^{(p-1)/2}}\right\}\d x+\frac{\beta_p^2}{N^{p-1}}\E\xi_{j,a,\mathbf k}^2\right)
\end{align*}
for some constant $C$ depending only on $n$, and the right-hand side tends to zero. (Recall the explanation below (\ref{diffF2}).) The last integral of (\ref{hxig-diff}) can be handled similarly. It tends to zero upon applying (\ref{IBP2}) to each $\E\langle r_{N,\xi}(\sspin^\ell)F\rangle_t$.
Indeed, for some constant $C'$ depending only on $(\beta_p)$ and $n$, we have
\[
\sup_{p\geq 2,N\geq 1}|\ms I(p)| N^{-(p-2)}<\infty\Longrightarrow \sup_{\sspin^1,\sspin^2\in \Sigma_{N+n'}}\big|\E[r_{N,\xi}(\sspin^1)r_{N,\xi}(\sspin^2)]\big|\leq \frac{C'}{N}, 
\]
and the remainder term $\mathcal E_\ell(t)$ from the application of (\ref{IBP2}) to $\E\langle r_{N,\xi}(\sspin^\ell)F\rangle_t$ satisfies
\begin{align*}
|\mathcal E_\ell(t)|\leq C'\sum_{p\geq 2}
\sum_{\mathbf i\in \ms I(p)}
\frac{\beta_p^2}{N^{p-1}}\left(\E|\xi_{\mathbf i}|\int_{0}^{|\xi_{\mathbf i}|}\min\left\{1,\frac{|\beta_p |x}{N^{(p-1)/2}}\right\}dx+\frac{|\beta_p|^2}{N^{p-1}}\E\xi_{\mathbf i}^2\right)\xrightarrow[N\to\infty]{} 0.
\end{align*}
Apply the last three displays to (\ref{hxig-diff}), and then (\ref{eq:decg3}) follows. \\

\noindent \hypertarget{step2-4}{{\bf Step 4.}} 
From this step on, we draw connections between universality of joint distributions of the overlap arrays $(R_{\ell,\ell'})_{\ell\neq \ell'}$ and the required universality of spin distributions.

By Step~\hyperlink{step2-3}{3},
the argument in \cite[Section~3]{Panchenko_2013_spin} now applies and gives the following results. First, for any $n,n'\geq 1$ and $C_1,\cdots,C_n\subseteq \{1,\cdots,n'\}$, the following identity holds:
\begin{align}\label{eq:cavityg}
\E\prod_{\ell\leq n}\left\langle \prod_{j\in C_\ell}\vep_j\right\rangle_{H^0_{N+n',\xi,g}} =\E\prod_{\ell\leq n}\frac{\big\langle \prod_{j\in C_\ell}\tanh\big(h_{N,j,g}(\spin)+h\big)\prod_{j=1}^{n'} \cosh\big(h_{N,j,g}(\spin)+h\big)\big\rangle_{H'_{N,\xi}}}{\big\langle \prod_{j=1}^{n'}\cosh\big(h_{N,j,g}(\spin)+h\big)\big\rangle_{H'_{N,\xi}}}.
\end{align}
The key feature of the right-hand side of (\ref{eq:cavityg}) is the presence of the Gaussian Hamiltonians $h_{N,j,g}$. 
Second, the key observation in \cite[Section~3]{Panchenko_2013_spin} applies to the present case and shows
that, through (\ref{eq:cavityg}), for every $\vep>0$, we can find a positive integer $M_\vep$ and a bounded continuous function $\Phi_\vep$, both independent of $N$ and the disorder $\xi$, such that
\begin{align}\label{SW:approx}
&\Bigg| \E\prod_{\ell\leq n}\left\langle \prod_{j\in C_\ell}\vep_j\right\rangle_{H^0_{N+n',\xi,g}}-\E\left\langle \Phi_\vep\left(\frac{N}{N+n'}R_{\ell,\ell'};1\leq \ell\neq \ell'\leq M_\vep\right)\right\rangle_{H'_{N,\xi}}\Bigg|\leq \vep.
\end{align}

\noindent \hypertarget{step2-5}{{\bf Step 5.}} We complete the proof of Theorem~\ref{thm:2} in this step. First, notice that
the limiting distribution of the overlap arrays (\ref{def:overlap}) under $\E\langle\, \cdot\,\rangle_{H'_{N,\xi}}$ is given by the probability distribution $\Pi_\xi$ defined in Step~\hyperlink{Step 2-1}{1}. In more detail, we can work with free-energy densities under the following interpolating Hamiltonians: 
\begin{align}\label{HN:decg4}
\widetilde{H}_{N+n',\xi,g}^t(\sspin)=\sqrt{1-t}H'_{N,\xi}(\spin)+\sqrt{t}H_{N,\xi}(\spin)+\sum_{j=1}^{n'}\vep_j h_{N,j,g}(\spin),\quad 0\leq t\leq 1,
\end{align}
to tune the normalizing factor $1/(N+n')^{(p-1)/2}$ in $H'_{N,\xi}$ to $1/N^{(p-1)/2}$,
where $H_{N,\xi}$ is an independent copy of the Hamiltonian of the generic $p$-spin model over $N$ sites. Then it can be shown that the rates of change of the free-energy densities associated with $\widetilde{H}_{N+n',\xi,g}^t$ vanish
by Proposition~\ref{prop:AIBP} and the fact that, for $p\geq 2$, 
\begin{align}\label{beta+g}
\left|\frac{N^{p-1}}{(N+n')^{p-1}}\beta_p^2-\beta_p^2\right|=\mathcal O\left(\frac{1}{N}\right).
\end{align}
We omit the details since they are very similar to the arguments in Step~\hyperlink{step2-1}{1} and Step~\hyperlink{step2-3}{3}.  

Now recall that by Step~\hyperlink{Step 2-1}{1}, the laws of $(R_{\ell,\ell'})_{\ell\neq \ell'}$ under $\Pi_\xi$ and $\Pi_g$ are the same. 
Hence by (\ref{eq:decg3}) and (\ref{SW:approx}), 
$\lim_{N\to\infty}\E\prod_{\ell\leq n}\big\langle \prod_{j\in C_\ell}\vep_j\big\rangle_{H_{N+n',\xi}}$
exists and the limits are the same for all disorders where the first two moments match the corresponding moments of a standard Gaussian. The proof is complete.
\end{proof}

In the next result,
we show some identities satisfied by 
\begin{align}\label{normalized sum}
\frac{1}{N^2}\sum_{\i\in \{1,\cdots,N\}^2}\E\langle \D^m_\i F_N\rangle_{H_{N,g}},\quad 3\leq m\leq 4,
\end{align}
in the limit of large $N$ for appropriate functions $F_N$, where the operators $\D_\i$'s are defined by (\ref{def:Dij}). Notice that in (\ref{normalized sum}), our interest is in the use of the standard Gaussian disorder. The proof will show that  
the universality in Theorem~\ref{thm:2} can be used to elicit identities which seem to be new to these spin glass models, whereas the models are traditionally studied under the Gaussian disorder. In Section~\ref{sec:D4}, we will study more systematically Gibbs expectations of the form $\langle \D^m_\i F\rangle_{H_{N,g}}$, and then we will explain how the normalized sums in (\ref{normalized sum}) can be reformulated in terms of multi-overlaps which are generalizations of overlaps.

\begin{prop}\label{prop:1}
Let a sequence $(\beta_p)_{p\geq 2}$ satisfying (\ref{def:generic}) and $\beta_2\neq 0$ be given.
For fixed $n,n'\geq 1$ and subsets $C_1,\cdots,C_n$ of $ \{1,\cdots,n'\}$,
define the following functions:
\begin{align}\label{def:SN}
S_N(\spin^1,\cdots,\spin^n)=\prod_{\ell=1}^n\prod_{j\in C_\ell} \sigma^\ell_j, \quad \spin^1,\cdots,\spin^n\in \Sigma^n_N. 
\end{align}
Then we have
\begin{align}\label{limit of SN}
\lim_{N\to\infty}\frac{1}{N^2}\sum_{\i\in \{1,\cdots,N\}^2}\E\langle \D^m_\i S_N\rangle_{H_{N,g}}=0,\quad 3\leq m\leq 4.
\end{align}
\end{prop}
\begin{proof}[\bf Proof.]
The main inputs of this proof are the following two sets of coupling constants $\{\xi_{\mathbf i}\}$ and $\{\eta^N_{\mathbf i}\}$ for the mixed $p$-spin model over $N$ sites. They are to be compared with a family of i.i.d. standard Gaussian variables $\{g_{\mathbf i}\}$. 

First,
take i.i.d. coupling constants $\xi_{\mathbf i}$ such that their first three moments match the corresponding moments of a standard Gaussian, but their fourth moments are \emph{different} from $\E[g_\i^4]= 3$. Second, we take independent (but not identically distributed) coupling constants $\eta^{N}_\i$  such that
\begin{enumerate}
\item [\rm (1)] $\eta^N_\i\stackrel{\rm (d)}{=}g_\i$ for all $\i\notin \{1,\cdots,N\}^2$, and
\item [\rm (2)] the first two moments of $\eta^N_\i$, for $\i\in \{1,\cdots,N\}^2$, match the corresponding moments of a standard Gaussian, $\E (\eta^N_\i)^3=1/N^{1/2}$, and $|\E (\eta^N_\i)^4-3|\leq C/N^{2/3}$ for a universal constant $C$. 
\end{enumerate}
Here, condition (2) is the key property which we require. To meet this condition,  we choose those variables $\eta^N_\i$ by suitable perturbations of i.i.d. standard Gaussians as follows. Suppose that $\{\zeta_{\mathbf i}\}_{\mathbf i\in \{1,\cdots,N\}^2}$ are i.i.d. and independent of i.i.d. standard Gaussians $\{g_{\mathbf i}'\}_{\mathbf i\in \{1,\cdots,N\}^2}$ such that $\E[\zeta^3_{\mathbf i}]=1$, $\E[\zeta_{\mathbf i}^4]<\infty$ and the first two moments of $\zeta_{\mathbf i}$ match the corresponding moments of $g'_{\mathbf i}$. See (\ref{def:inter_disorder}). Then the first two moments of 
$\eta^N_{\mathbf i}=\sqrt{1/N^{1/3}}\zeta_{\mathbf i}+\sqrt{1-N^{1/3}}g_{\mathbf i}'$ has zero mean and unit variance, and we have
\begin{align*}
\E[(\eta^N_{\mathbf i})^3]=\frac{1}{N^{1/2}}\E[\zeta_{\mathbf i}^3]=\frac{1}{N^{1/2}}\quad\mbox{and}\quad  
\E[(\eta^N_{\mathbf i})^4]
=\frac{1}{N^{2/3}}(\E[\zeta_{\mathbf i}^4]-3)+3.
\end{align*}
We assume that the three families $\{g_{\mathbf i}\}$, $\{\xi_{\mathbf i}\}$ and $\{\eta^N_{\mathbf i}\}$ are independent.

We first prove (\ref{limit of SN}) with $m=4$ in Step~\hyperlink{step3-1}{1}--Step~\hyperlink{step3-3}{3}
by comparing $\{g_{\mathbf i}\}$ and $\{\xi_{\mathbf i}\}$. Then  we prove (\ref{limit of SN})  with $m=3$ in Step~\hyperlink{step3-4}{4} by comparing $\{g_{\mathbf i}\}$ and $\{\eta^N_{\mathbf i}\}$. The assumptions on $(\beta_p)$ are only used in Step~\hyperlink{step3-3}{3} and Step~\hyperlink{step3-4}{4}.
\\

\noindent \hypertarget{step3-1}{{\bf Step 1.}}
We revisit the interpolation method applied in \cite[Theorem~4.3]{AC_15} to Gibbs measures in this step and Step~\hyperlink{step3-2}{2}.
For $0\leq t\leq 1$, define disorder $\xi^t$ by the following equation:
\begin{align}\label{def:inter_disorder}
\xi^t_{\mathbf i}\doteq \sqrt{t}\xi_{\mathbf i}+\sqrt{1-t}g_{\mathbf i}
\end{align}
and write the corresponding Gibbs expectation as $\langle \,\cdot\,\rangle_t$. 
Let $\langle \,\cdot\,\rangle_{t,\xi_{\mathbf i}=x}$ be obtained from $\langle \,\cdot\,\rangle_{t}$ by replacing $\xi_{\mathbf i}$ with $x$, and $f_{\mathbf i}(t,x)\equiv \E\langle 
\D_{\mathbf i}F\rangle_{t,\xi_{\mathbf i}=x}$. 
Then for any real-valued function $F$ defined on $\Sigma^n_N$ and $T\in [0,1]$, we deduce  from Proposition~\ref{prop:AIBP} and the Gaussian integration by parts (as in the treatment of the first two terms on the right-hand side of (\ref{hxig-diff})) that
\begin{align}
\E\langle F\rangle_T-\E\langle F\rangle_0=&\int_0^T\E\frac{\d }{\d t}\langle F\rangle_t \d t=\int_0^T\sum_{p=2}^\infty \Bigg[\frac{\beta_p}{2N^{(p-1)/2}}\sum_{\mathbf i\in \{1,\cdots,N\}^p}\frac{\error^2_{\xi_{\mathbf i}}\big(\partial_x^2f_{\mathbf i}(t,\cdot)\big)}{\sqrt{t}} \Bigg]\d t.
\label{1-0}
\end{align}

\noindent \hypertarget{step3-2}{{\bf Step 2.}}
In this step, we show that (\ref{1-0}) implies 
\begin{align}
\begin{split}
\label{univ}
\E\langle F\rangle_T-\E\langle F\rangle_{0}
=\beta_2^4\left(\frac{\E\xi_{\mathbf k}^4-3}{12}\right)\int_{0}^T t\Bigg(\frac{1}{N^2} \sum_{\mathbf i\in \{1,\cdots,N\}^2}\E\langle \D^4_{\mathbf i}F\rangle_t\Bigg) \d t+\int_0^T\mathcal E_t(F)\d t
\end{split}
\end{align}
for any $\mathbf k\in \{1,\cdots,N\}^2$,
where $\mathcal E_t(F)$ converges to zero as $N\to\infty$ uniform in $t\in (0,1]$ and $F:\Sigma^n_N\to [-1,1]$.

For any $p\geq 2$ and $\mathbf i\in \{1,\cdots,N\}^p$, set 
\begin{align}
\begin{split}\label{def:vepi}
\vep_{\mathbf i}(t)=&\E\frac{\xi_{\mathbf i}}{2}\int_0^{\xi_{\mathbf i}}(\xi_{\mathbf i}-x)^2\big(\partial^3_xf_{\mathbf i}(t,x)-\E\partial^3_xf_{\mathbf i}(t,\xi_{\mathbf i})\big)\d x\\
&\hspace{3cm}-\E\int_0^{\xi_{\mathbf i}}(\xi_{\mathbf i}-x)\big(\partial^3_xf_{\mathbf i}(t,x)-\E\partial^3_xf_{\mathbf i}(t,\xi_{\mathbf i})\big)\d x.
\end{split}
\end{align}
Then it follows from (\ref{gamma2}) that 
\begin{align}
\error^2_{\xi_{\mathbf i}}\big(\partial^2_xf_{\mathbf i}(t,\cdot)\big)=&\E\xi_{\mathbf i}\int_0^{\xi_{\mathbf i}}(\xi_{\mathbf i}-x)\partial^2_xf_{\mathbf i}(t,x)\d x-\E\int_0^{\xi_{\mathbf i}}(\xi_{\mathbf i}-x)\partial^3_xf_{\mathbf i}(t,x)\d x\notag\\
=&\left(\frac{\E \xi_{\mathbf i}^3}{2}\right)\cdot\E\partial_x^2f_{\mathbf i}(t,0)+\E\frac{\xi_{\mathbf i}}{2}\int_0^{\xi_{\mathbf i}}(\xi_{\mathbf i}-x)^2\partial^3_xf_{\mathbf i}(t,x)\d x-\E\int_0^{\xi_{\mathbf i}}(\xi_{\mathbf i}-x)\partial^3_xf_{\mathbf i}(t,x)\d x\notag\\
=&\left(\frac{\E \xi_{\mathbf i}^3}{2}\right)\cdot\E\partial_x^2f_{\mathbf i}(t,0)+\left(\frac{\E\xi_{\mathbf i}^4-3\E \xi_{\mathbf i}^2}{6}\right)\cdot \E\partial^3_xf_{\mathbf i}(t,\xi_{\mathbf i})+\vep_{\mathbf i}(t)\label{gamma:mom}\\
=&\left(\frac{\E\xi_{\mathbf i}^4-3}{6}\right)\cdot \frac{(\sqrt{t}\beta_p)^3}{N^{3(p-1)/2}}\E\langle \D^4_{\mathbf i}F\rangle_t+\vep_{\mathbf i}(t).\label{gamma2-0}
\end{align}
Here, the second equality follows from the ordinary integration by parts and the fact that the coupling constants are independent,
and (\ref{gamma2-0}) uses the assumption $\E\xi_{\mathbf i}^2=1$ and $\E \xi_{\mathbf i}^3=0$ and the analogue of (\ref{Gibbs:der}) with $\beta_p$ replaced by $\sqrt{t}\beta_p$. 

To bound $\vep_{\mathbf i}(t)$ defined by (\ref{def:vepi}), we notice that,
by the same analogue of (\ref{Gibbs:der}) again and the mean-value theorem,
\begin{align*}
\left|\partial^3_xf_{\mathbf i}(t,x)-\E\partial^3_xf_{\mathbf i}(t,\xi_{\mathbf i})\right|\leq &\min\Big\{2\|\partial_x^3f_{\mathbf i}(t,\,\cdot\,)\|_\infty,\,\E[|\xi_{\mathbf i}-x|]\cdot  \|\partial_x^4f_{\mathbf i}(t,\,\cdot\,)\|_\infty\Big\}\\
\leq &\frac{C|\sqrt{t}\beta_p|^3}{N^{3(p-1)/2}}\min\left\{1,\frac{|\sqrt{t}\beta_p|}{N^{(p-1)/2}}\E[|\xi_{\mathbf i}|+|x|]\right\}
\end{align*}
for some constant $C$ depending only on $n$. In addition,
for any nonzero $\xi\in \R$, $k\in \Bbb N$ and bounded measurable function $f$, we have
\[
\left|\int_0^\xi (\xi-x)^k f(x)\d x\right|\leq |\xi|^k\int_{0}^{|\xi|} \big|f\big(\sgn(\xi)x\big)\big|\d x.
\]
Applying the last two displays to (\ref{def:vepi}), we get the following inequality:
\begin{align}\label{vep:bdd}
|\vep_{\mathbf i}(t)|
\leq &\E \left(\frac{|\xi_{\mathbf i}|^3}{2}+|\xi_{\mathbf i}|\right)\int_{0}^{|\xi_{\mathbf i}|}
\frac{C|\sqrt{t}\beta_p|^3}{N^{3(p-1)/2}}\min\left\{1,\frac{|\sqrt{t}\beta_p|}{N^{(p-1)/2}}\E[|\xi_{\mathbf i}|+|x|]\right\}
\d x,
\end{align}
where the right-hand side is finite since $\E[\xi_\i^4]<\infty$.

Now, since $4(p-1)/2-p=p-2$, it follows from  (\ref{gamma2-0}) and (\ref{vep:bdd}) that the partial sum of (\ref{1-0}) over $p\geq 3$ is  bounded by $\mathcal O(N^{-1})$, uniformly in $t\in (0,1]$ and $F:\Sigma^n_N\to [-1,1]$. Also, (\ref{vep:bdd}) shows
\[
\lim_{N\to\infty}\sup_{t\in (0,1]}\sup_{F:\Sigma^n_N\to [-1,1]}\frac{1}{N^{1/2}}\sum_{\mathbf i\in \{1,\cdots,N\}^2}|\vep_{\mathbf i}(t)|=0.
\]
The required equality (\ref{univ}) thus follows from (\ref{gamma2-0}) since $\xi_\i$ are i.i.d. \\

\noindent \hypertarget{step3-3}{{\bf Step 3.}}
Recall the functions $S_N$ defined by (\ref{def:SN}).
By (\ref{univ}) and Theorem~\ref{thm:2}, we get, for all $T\in [0,1]$ and $\k\in \{1,\cdots,N\}^2$,
\begin{align}
\lim_{N\to\infty}\beta_2^4\left(\frac{\E\xi_{\mathbf k}^4-3}{12}\right)\int_0^Tt\Bigg( \frac{1}{N^2}\sum_{\mathbf i\in \{1,\cdots,N\}^2}\E\langle \D^4_{\mathbf i}S_N\rangle_t\Bigg)\d t=\lim_{N\to\infty}\E\langle S_N\rangle_T-\E\langle S_N\rangle_0=0.\label{D4limit}
\end{align}
On the other hand,
 for any given sequence of functions $F_N:\Sigma^n_N\to [-1,1]$ for fixed $n$,
the functions 
$t\mapsto \E\langle F_N\rangle_t$ defined on $[0,1]$ are equicontinuous by (\ref{univ}). Since $\beta_2\neq 0$ and $\E\xi_{\mathbf k}^4\neq 3$, (\ref{D4limit}) implies
\begin{align}\label{D4-conv}
\lim_{N\to\infty}\frac{1}{N^2}\sum_{\mathbf i\in \{1,\cdots,N\}^2}\E\langle \D^4_{\mathbf i}S_N\rangle_t=0,\quad \forall\;0\leq t\leq 1.
\end{align}
Setting $t=0$ in (\ref{D4-conv}) gives (\ref{limit of SN}) with $m=4$.\\  

\noindent \hypertarget{step3-4}{{\bf Step 4.}} The proof of (\ref{limit of SN}) with $m=3$ is similar. Now with the disorder $\eta^N$ chosen at the beginning of this proof, we work with the interpolating coupling constants
\begin{align}\label{def:etaNtt}
\eta^{N,t}_\i=\sqrt{t}\eta^N_\i+\sqrt{1-t}g_\i,\quad 0\leq t\leq 1,
\end{align}
and write the corresponding Gibbs expectations as $\langle \,\cdot\,\rangle_{N,t}$. By the choice of $\eta^N_\i$, the arguments in Step~\hyperlink{step3-1}{1} and Step~\hyperlink{step3-2}{2} (see (\ref{1-0}), (\ref{gamma:mom}) and (\ref{vep:bdd}) in particular) can be modified slightly to show that, for all 
$F:\Sigma^n_N\to [-1,1]$, $\mathbf k\in \{1,\cdots,N\}^2$ and $T\in [0,1]$,
\begin{align}
\begin{split}
\E\langle F\rangle_{N,T}-\E\langle F\rangle_{N,0}=&\beta_2^3
\left(\frac{\E (\eta^N_\k)^3}{4}\right)\int_0^Tt^{1/2}
\Bigg(\frac{1}{N^{3/2}} \sum_{\mathbf i\in \{1,\cdots,N\}^2}\E\langle \D^3_{\mathbf i}F\rangle_{N,t}\Bigg) \d t\\
&\hspace{-1.5cm}+
\beta_2^4
\left(\frac{\E (\eta^N_\k)^4-3}{12}\right)\int_0^Tt
\Bigg(\frac{1}{N^2} \sum_{\mathbf i\in \{1,\cdots,N\}^2}\E\langle \D^4_{\mathbf i}F\rangle_{N,t}\Bigg) \d t
+\int_0^T\mathcal E_{N,t}(F)\d t,
\end{split}
\end{align}
where $\mathcal E_{N,t}(F)$ converges to zero as $N\to\infty$ uniform in $t\in (0,1]$.
See also the explanation below (\ref{diffF2}) and note that 
\begin{align}\label{etaggg}
|\eta^N_{\mathbf i}|\leq |\zeta_{\mathbf i}|+|g'_{\mathbf i}|\quad\mbox{ for $\mathbf i\in \{1,\cdots,N\}^2$} 
\end{align}
when it comes to obtain analogues of (\ref{vep:bdd}). 

An inspection of the proof of Theorem~\ref{thm:1} shows that its main result still applies when the  coupling constants in (\ref{def:etaNtt})  for fixed $t\in [0,1]$, which may not be identically distributed, are in force. Indeed,  in modifying Steps~\hyperlink{step2-3}{3} and~\hyperlink{step2-5}{5} of that proof, we can use (\ref{etaggg}) to handle the error terms from the generalized Gaussian integration by parts (Proposition~\ref{prop:AIBP}).
Then by the foregoing equality and the particular choice of $\eta^N_\i$ for $\i\in \{1,\cdots,N\}^2$, we can argue as in Step~\hyperlink{step3-3}{3} to obtain
(\ref{D4-conv}) with $\E\langle \D^4_\i S_N\rangle_t$ replaced by $\E\langle \D^3_\i S_N\rangle_{N,t}$. Taking $t=0$ leads to (\ref{limit of SN}) with $m=3$.
The proof is complete.
\end{proof}

\section{Gibbs measures and multi-overlaps}\label{sec:D4}
In the previous sections, derivatives of the Gibbs measures with respect to coupling constants by (\ref{Gibbs:der}) are used to study their universality and elicit the identities in Proposition~\ref{prop:Dj}. In any case, Gibbs expectations of the form $\langle \D^m_\i F\rangle$, for integers $n,N\geq 1$ and $F$ defined on $\Sigma^n_N$, are consistently present, where the operators $\D_\i$ are defined in (\ref{def:Dij}).

Our goal in this section is to study these particular Gibbs expectations $\langle \D^m_\i F\rangle$ combinatorially. Here and throughout the rest of this section, $\langle \,\cdot\,\rangle$ is  defined by the sum of a general mixed $p$-spin Hamiltonian as in (\ref{def:H}) and an arbitrary Hamiltonian over $N$ sites. In Proposition~\ref{prop:Dj}, we will show  that these Gibbs expectations  $\langle \D^m_\i F\rangle$ satisfy certain power-series-like expansions. This result will be applied in Section~\ref{sec:convtv}.

Let us introduce the two major sets of ingredients for the series expansion of $\langle \D^m_\i F\rangle$.
First, the series expansion  is a linear combination of Gibbs expectations of the following functions defined on $\Sigma^\infty_N$: for  $n\in \Bbb N$, $m\in \Bbb Z$ and $\i\in \{1,\cdots,N\}^p$ with $p\geq 2$, \begin{align}\label{def:Sj}
S^{m,n}_\i\doteq \left\{
\begin{array}{ll}
\displaystyle m!\sum_{k=0}^m(-1)^{m-k}{n+m-k-1\choose n-1}\sum_{1\leq \ell_1<\cdots<\ell_k\leq n}\sigma_{\mathbf i}^{\ell_1,\cdots,\ell_k,n+1,\cdots,n+m-k},& m\geq 1,\\
\vspace{-.2cm}\\
0, &m\leq 0,
\end{array}
\right.
\end{align}
where 
\[
\sigma^{\l_1,\cdots,\ell_m}_{\mathbf i}=\prod_{j=1}^m\sigma^{\ell_j}_{\mathbf i}
\]
and
the following convention for summation is used:  
\begin{align}
&\sum_{1\leq \ell_1<\cdots<\ell_k\leq n}\sigma_{\mathbf i}^{\ell_1,\cdots,\ell_k,n+1,\cdots,n+m-k}=\left\{
\begin{array}{ll}
\sigma_{\mathbf i}^{n+1,\cdots,n+m},&k=0,\\
0,&k>n.
\end{array}
\right.\label{convention:Dj}
\end{align}
Second, the series expansion uses coefficients given by the real constants $A(m,a)$, for $m\in \Bbb N $ and $a\in \Bbb Z_+$, which are defined as follows: 
\begin{align}\label{def:Aj0}
A(m,0)= 0,\quad m\geq 1,
\end{align}
and $A(m,a)$ for $m,a\geq 1$ are recursively defined by the following partial difference equation:
\begin{align}\label{def:Ajd}
\left\{
\begin{array}{ll}
A(1,a)&=\1_{\{1\}}(a),\\
A(m+1,a)&=-(m-2a+4)(m-2a+3)A(m,a-1) +A(m,a),\quad m,a\geq 1. 
\end{array}
\right.
\end{align}
For example, it is readily checked that
\begin{align}\label{eg:A}
\begin{split}
&A(2,1)=1,\;A(2,a)=0,\;a\geq 2;\\
&A(3,1)=1,\;A(3,2)=-2,\;A(3,a)=0,\;a\geq 3;\\
&A(4,1)=1, \; A(4,2)=-8,\;A(4,a)=0,\;a\geq 3.
\end{split}
\end{align}
The recursive definition (\ref{def:Ajd}) implies that $a\mapsto A(m,a)$ has a finite support. Precisely, we have the following.

\begin{lem}\label{lem:Asupp}
For any $m\in \Bbb N$, 
\begin{align}\label{A:supp}
A(m,a)=0,\quad\forall\; a\geq \left\lceil \frac{m}{2}\right\rceil +1,
\end{align}
where $\lceil x\rceil$ is the smallest integer $\geq x$. 
\end{lem}
\begin{proof}[\bf Proof.]
We claim that (\ref{A:supp}) holds  for all $m\in \{2p-1,2p\}$ by an induction on $p\in \Bbb N$.  The case $p=1$ already follows from (\ref{def:Ajd}) and (\ref{eg:A}). Suppose that (\ref{A:supp}) holds for some $p\geq 1$. First, we consider $m=2p+1$. For $a_0= \lceil \frac{m}{2}\rceil+1=p+2$, the second line in the definition (\ref{def:Ajd}) implies that
$A(m,a_0)=0$ since $(m-1)-2a_0+4=0$ and the following with $a$ replaced by $a_0$ is satisfied:
\begin{align}\label{A:supp0}
a\geq \left\lceil \frac{m-1}{2}\right\rceil+1=p+1
\end{align} 
so that our hypothesis for the value $p$ applies to obtain $A(m-1,a_0)=0$. To see that $A(m,a_0)=A(2p+1,a_0)=0$ for any $a_0>\lceil \frac{m}{2}\rceil+1=p+2$, notice that by (\ref{def:Ajd}), $A(2p+1,a_0)$ is given by a linear combination of $A(2p,a_0-1)$ and $A(2p,a_0)$, which are both equal to zero again by our hypothesis for the value $p$ since (\ref{A:supp0}) holds for $a=a_0-1$ and $a_0$. We have proved that (\ref{A:supp}) holds for $m=2p+1$. 

For the other case that $m=2p+2$, the argument above still applies to obtain (\ref{A:supp}), except that now we use the equation $(m-1)-2a_1+3=0$, when $a_1=\left\lceil \frac{m}{2}\right\rceil+1 =p+2$, and the condition in (\ref{A:supp0}) with the lower bound replaced by $\lceil \frac{m-1}{2}\rceil+1=p+2$. In summary, we have proved (\ref{A:supp}) for $m\in \{2p+1,2p+2\}$.
By induction, (\ref{A:supp}) holds for all $m\geq 1$. 
\end{proof}

We are ready to state the series expansion for $\langle \D^m_\mathbf iF\rangle$.

\begin{prop}\label{prop:Dj}
Assume that the underlying Hamiltonian $H_{N,\xi}(\spin)$ defining the Gibbs measure in (\ref{Gibbs}) is replaced by the sum of a general mixed $p$-spin Hamiltonian as in (\ref{def:H}) and an arbitrary Hamiltonian over $N$ sites.
With the functions $S^{m,n}_\i$ defined in (\ref{def:Sj}) and the constants $A(m,a)$ defined above in (\ref{def:Aj0}) and (\ref{def:Ajd}), it holds that
\begin{align}\label{Dj}
\begin{split}
\langle \D^m_\mathbf iF\rangle=\sum_{a=1}^\infty A(m,a)\langle S_\i^{m-2a+2,n} F\rangle 
\end{split}
\end{align}
for all $n,m\geq 1$, $p\geq 2$, $\mathbf i\in \{1,\cdots,N\}^p$, and $F:\Sigma^n_N\to [-1,1]$.
\end{prop}
\begin{proof}[\bf Proof.]
By continuity, we may assume that $\beta_p\neq 0$. We follow the convention for summation as in (\ref{convention:Dj}) in this proof. 

As the first step to prove (\ref{Dj}), we derive an ordinary differential equation satisfied by $\langle S^{m,n}_\i F\rangle$ for any $m\geq 1$. We use (\ref{Gibbs:der}) and calculate the following:  
\begin{align}
&\frac{N^{(p-1)/2}}{\beta_p}\partial_{\xi_\i}\la S_\i^{m,n}F\ra\notag \\
=& m!\sum_{k=0}^m(-1)^{m-k}{n+m-k-1\choose n-1}\sum_{1\leq \ell_1<\cdots<\ell_k\leq n}\frac{N^{(p-1)/2}}{\beta_p}\partial_{\xi_\i}\langle  \sigma_{\mathbf i}^{\ell_1,\cdots,\ell_k,n+1,\cdots,n+m-k}F\rangle
\notag\\
=&m!\sum_{k=0}^m(-1)^{m-k}{n+m-k-1\choose n-1}\notag\\
&\times\sum_{1\leq \ell_1<\cdots<\ell_k\leq n}\left(\sum_{\ell=1}^{n+m-k}\langle \sigma_{\mathbf i}^{\ell_1,\cdots,\ell_k,n+1,\cdots,n+m-k,\ell} F\rangle-(n+m-k)\langle \sigma_\i^{\ell_1,\cdots,\ell_k,n+1,\cdots,n+m-k+1}F\ra \right)\notag\\
=&m!\sum_{k=0}^m(-1)^{m-k}{n+m-k-1\choose n-1}
\Bigg(\sum_{1\leq \ell_1<\cdots<\ell_k\leq n}
\sum_{\stackrel{\scriptstyle \ell\notin \{\ell_1,\cdots,\ell_k\}}{1\leq \ell\leq n}}
\langle  \sigma_{\mathbf i}^{\ell_1,\cdots,\ell_k,n+1,\cdots,n+m-k,\ell} F\rangle\notag\\
&-(n+m-k)\sum_{1\leq \ell_1<\cdots<\ell_k\leq n}\langle \sigma_\i^{\ell_1,\cdots,\ell_k,n+1,\cdots,n+m-k+1}F\ra \Bigg)\notag
\\
&+m!\sum_{k=0}^m(-1)^{m-k}{n+m-k-1\choose n-1}
\sum_{1\leq \ell_1<\cdots<\ell_k\leq n}
\sum_{\ell\in \{\ell_1,\cdots,\ell_k\}}
\langle  \sigma_{\mathbf i}^{\ell_1,\cdots,\ell_k,n+1,\cdots,n+m-k,\ell}F\rangle\notag\\
&+m!\sum_{k=0}^m(-1)^{m-k}{n+m-k-1\choose n-1}
\sum_{1\leq \ell_1<\cdots<\ell_k\leq n}
\sum_{\ell=n+1}^{n+m-k}
\langle  \sigma_{\mathbf i}^{\ell_1,\cdots,\ell_k,n+1,\cdots,n+m-k,\ell}F\rangle\notag\\
=&{\rm I}_m+{\rm II}_m,\label{I and II} 
\end{align}
where ${\rm I}_m$ is defined by the first sum in the next to the last equality and ${\rm II}_m$ is defined by the sum of the second and third sums there. 
Notice that each of the Gibbs expectations in ${\rm I}_m$ contains a product of $\sigma^\ell_\i$ for
$(n+m+1)$ many distinct $\ell$'s. But for the Gibbs expectations in ${\rm II}_m$, there are only products of $\sigma^\ell_\i$ for $(n+m-1)$ many distinct $\ell$'s due to the `killing effect' that $\sigma^{\ell,\ell}_\i=1$.  

We can simplify the term ${\rm I}_m$ as follows. If $m-k\geq 1$,  then
\begin{align}\label{comb:1}
&{n+m-k-1\choose n-1}(k+1)+{n+m-k-2\choose n-1}(n+m-k-1)
={n+m-k-1\choose n-1}(m+1).
\end{align}
Thus, keeping in mind the convention for summation as in (\ref{convention:Dj}), we obtain
\begin{align}
{\rm I}_m=&m!\sum_{k=0}^m(-1)^{m-k}{n+m-k-1\choose n-1}(k+1)\sum_{1\leq \ell_1<\cdots<\ell_{k+1}\leq n}
\langle  \sigma_{\mathbf i}^{\ell_1,\cdots,\ell_{k+1},n+1,\cdots,n+m-k}F\rangle\notag\\
&+m!\sum_{k=-1}^{m-1}(-1)^{m-k}{n+m-k-2\choose n-1}(n+m-k-1)\sum_{1\leq \ell_1<\cdots<\ell_{k+1}\leq n}\langle \sigma_\i^{\ell_1,\cdots,\ell_{k+1},n+1,\cdots,n+m-k}F\ra\notag\\
=&m!{n-1\choose n-1}(m+1)\sum_{1\leq \ell_1<\cdots<\ell_{m+1}\leq n}
\langle  \sigma_{\mathbf i}^{\ell_1,\cdots,\ell_{m+1}}F\rangle\notag\\
&+m!\sum_{k=0}^{m-1}(-1)^{m-k}\left[{n+m-k-1\choose n-1}(k+1)+{n+m-k-2\choose n-1}(n+m-k-1)\right]\notag\\
&\times \sum_{1\leq \ell_1<\cdots<\ell_{k+1}\leq n}\langle \sigma_\i^{\ell_1,\cdots,\ell_{k+1},n+1,\cdots,n+m-k}F\ra\notag\\
&+m!(-1)^{m+1}{n+m-1\choose n-1}(n+m)\langle \sigma_\i^{n+1,\cdots,n+m+1}F\ra\notag\\
=& (m+1)!\sum_{k=0}^{m+1}(-1)^{m+1-k}{n+(m+1)-k-1\choose n-1}\sum_{1\leq \ell_1<\cdots<\ell_k\leq n}\langle  \sigma_{\mathbf i}^{\ell_1,\cdots,\ell_k,n+1,\cdots,n+(m+1)-k}F\rangle\notag\\
=&\la S_\i^{m+1,n}F\ra, \label{I}
\end{align}
where the next to the last equality uses (\ref{comb:1})
and the last equality follows from the definition of the function $S^{m+1,n}_\i$.

For the term ${\rm II}_m$, we use the following identities for binomial coefficients: for $m-k\geq 1$,
\begin{align*}
&{n+m-k-2\choose n-1 }(n-k)-{n+m-k-1\choose n-1}(m-k)
=-{n+m-k-2\choose n-1}(m-1). 
\end{align*}
Then  we get
\begin{align}
{\rm II}_m=&m!\sum_{k=1}^m(-1)^{m-k}{n+m-k-1\choose n-1}(n-k+1)
\sum_{1\leq \ell_1<\cdots<\ell_{k-1}\leq n}
\langle  \sigma_{\mathbf i}^{\ell_1,\cdots,\ell_{k-1},n+1,\cdots,n+m-k}F\rangle\notag\\
&+m!\sum_{k=0}^{m-1}(-1)^{m-k}{n+m-k-1\choose n-1}(m-k)
\sum_{1\leq \ell_1<\cdots<\ell_k\leq n}
\langle  \sigma_{\mathbf i}^{\ell_1,\cdots,\ell_k,n+1,\cdots,n+m-k-1}F\rangle\notag\\
=&-(m-1)m!\sum_{k=0}^{m-1}(-1)^{(m-1)-k}{n+(m-1)-k-1\choose n-1}\sum_{1\leq \ell_1<\cdots<\ell_k\leq n}\langle \sigma_{\mathbf i}^{\ell_1,\cdots,\ell_k,n+1,\cdots,n+(m-1)-k}F\rangle\notag\\
=&-m(m-1)\la S_\i^{m-1,n}F\rangle\label{II}
\end{align}
by the definition of $S^{m-1,n}_\i$. 
In summary, by (\ref{I and II}), (\ref{I}) and (\ref{II}), the following differential equations hold for all $m\geq 1$:
\begin{align}\label{S:recursive}
\frac{N^{(p-1)/2}}{\beta_p}\partial_{\xi_\i}\la S_\i^{m,n}F\ra =-m(m-1)\la S_\i^{m-1,n}F\ra 
+\la S_\i^{m+1,n}F\ra .
\end{align}
Notice that the same differential equations are trivially satisfied for all $m\leq -1$ by (\ref{def:Sj}).

Now we prove (\ref{Dj}) by an induction on $m\geq 1$. 
For $m=1$, the required formula (\ref{Dj}) holds by writing out the summands of $S^{1,n}_\i$ with $k$ in the increasing order and recalling the definition (\ref{def:Dij}) of $\D_{\i}$. Next, suppose that (\ref{Dj}) holds up to some $m\geq 1$. Since $\beta_p\neq 0$, we can generate $\langle \D^{m+1}_\i F\rangle$ from $\langle\D^m_\i F\rangle$ by (\ref{Gibbs:der}) as follows:
\begin{align}\label{Dj+1sum}
&\langle \D_\i^{m+1} F\rangle=\frac{N^{(p-1)/2}}{\beta_p}\partial_{\xi_\i}\langle \D^m_\i F\rangle =\sum_{a=1}^\infty A(m,a)\left(\frac{N^{(p-1)/2}}{\beta_p}\partial_{\xi_\i}\la S_\i^{m-2a+2,n}F\ra \right),
\end{align}
where the second equality is justified by the fact that the sum in (\ref{Dj}) for $\langle \D^m_\i F\rangle$ is a finite sum. By the differential equations in (\ref{def:Sj}), which are valid for all nonzero $m\in \Bbb Z$, and the fact that $A(m,a)=0$ if $m-2a+2=0$ by (\ref{A:supp}), we see that (\ref{Dj+1sum}) gives
\begin{align*}
\langle \D_\i^{m+1} F\rangle=&\sum_{a=1}^\infty A(m,a)\big[-(m-2a+2)(m-2a+1)\la S_\i^{m-2a+1,n}F\ra 
+\la S_\i^{m-2a+3,n}F\ra \big]\\
=&\sum_{a=1}^\infty \big[-(m-2a+4)(m-2a+3)A(m,a-1)+A(m,a)\big] \la S_\i^{(m+1)-2a+2,n}F\ra \\
=&\sum_{a=1}^\infty A(m+1,a)\la S_\i^{(m+1)-2a+2,n} F\ra ,
\end{align*}
where the second equality uses (\ref{def:Aj0}) and the last equality follows from (\ref{def:Ajd}). We have proved that (\ref{Dj}) with $m$ replaced by $m+1$ holds. By mathematical induction, (\ref{Dj}) holds for all $m\geq 1$. The proof is complete.
\end{proof}

For $p\geq 2$ and $\mathbf i\in \{1,\cdots,N\}^p$,
we write 
\begin{align}\label{def:4overlap}
R_{\l_1,\cdots,\ell_m}\doteq \frac{1}{N}\sum_{i=1}^N\sigma^{\l_1,\cdots,\ell_m}_{i}
\end{align}
for the multi-overlaps of the spin configurations $\spin^{\ell_1},\cdots,\spin^{\ell_m}\in \Sigma_N$ so that
\begin{align}\label{sum:RR}
\frac{1}{N^2}\sum_{\i\in \{1,\cdots,N\}^2}\langle \sigma^{\ell_1,\cdots,\ell_m}_\i F\rangle =\langle R^2_{\ell_1,\cdots,\ell_m}F\rangle .
\end{align}
Therefore, the formula in (\ref{Dj}) can be expressed in terms of a formula in multi-overlaps whenever we sum over $\i\in \{1,\cdots,N\}^p$. The latter formula can be applied with Proposition~\ref{prop:1} to obtain identities of multi-overlaps.

\begin{eg}
By Proposition~\ref{prop:Dj} and (\ref{eg:A}), the following equation holds for any $n\in \Bbb N$ and any function $F$ defined $\Sigma^n_N$:
\begin{align*}
&\frac{1}{N^2}\sum_{\mathbf i\in \{1,\cdots,N\}^2}\E\langle \D^4_{\mathbf i}F\rangle\\
=&\,4!\sum_{
1\leq \ell_1<\ell_2<\ell_3<\ell_4\leq n}\E\langle R^2_{\ell_1,\ell_2,\ell_3,\ell_4}F\rangle-4!n\sum_{
1\leq \ell_1<\ell_2<\ell_3\leq n}\E\langle R^2_{\ell_1,\ell_2,\ell_3,n+1}F\rangle\\
&+12(n+1)n\sum_{
1\leq \ell_1<\ell_2\leq n}\E\langle R^2_{\ell_1,\ell_2,n+1,n+2}F\rangle-4(n+2)(n+1)n\sum_{\ell=1}^n\E\langle R^2_{\ell,n+1,n+2,n+3}F\rangle\\
&+(n+3)(n+2)(n+1)n
\E\langle R^2_{n+1,n+2,n+3,n+4}F\rangle\\
&-16\sum_{1\leq \ell_1<\ell_2\leq n}\E\langle R^2_{\ell_1,\ell_2}F\rangle+16n\sum_{\ell=1}^n\E\langle R^2_{\ell,n+1}F\rangle-8(n+1)n\E\langle R^2_{n+1,n+2}F\rangle.
\end{align*}
Notice that the right-hand side only involves overlaps and $4$-overlaps. 
\qed 
\end{eg}

\section{Total variation convergence of spin distributions}\label{sec:convtv}
In this section, we consider the mixed $p$-spin models subject to any choice of temperature parameters. Now
we perturb their Hamiltonians by independent Viana--Bray diluted $2$-spin glass Hamiltonians defined as follows. 
Let $\{\mathbf u_i,\mathbf v_i;i\geq 1\}$  be a family of i.i.d. variables uniformly distributed  over $\{1,\cdots,N\}$, $\{J_i;i\geq 1\}$ be a family of i.i.d. nonzero real-valued variables, and $X_{\alpha N}$ be a Poisson variable with mean $\alpha N$. All these variables are independent, and we assume that $J_i$'s are bounded and their first and third moments match the corresponding moments of a standard Gaussian.
Then the  Viana--Bray diluted spin glass Hamiltonian is defined by, for $\beta'\in \R$,
\begin{align}\label{def:H'}
H'_{\alpha,\beta'}(\spin)=H'_{N,\alpha,\beta'}(\spin)\equiv \beta'\sum_{i=1}^{ X_{\alpha N}}J_i\sigma_{\mathbf u_i,\mathbf v_i},\quad \spin \in \Sigma_N,
\end{align}
where $\sigma_{\mathbf u_i,\mathbf v_i}=\sigma_{(\mathbf u_i,\mathbf v_i)}=\sigma_{\mathbf u_i}\sigma_{\mathbf v_i}$.

\begin{thm}\label{thm:3}
Consider a mixed $p$-spin Hamiltonian defined by (\ref{def:H}) and the Viana--Bray diluted spin glass Hamiltonian as in (\ref{def:H'}). Assume that  
$\xi_{\mathbf i}$ have finite fourth moments and their first three moments match the corresponding moments of a standard Gaussian. The random variables $\xi_{\mathbf i}$ and $J_i$ may be distributed differently.

Let $\alpha_{k,j}\in (0,\infty)$ and $N_{k,j}\in  \Bbb N$ be such that  
\begin{align}\label{aN-hyp}
\lim_{k\to\infty}\alpha_{k,j}=0,\; \lim_{k\to\infty}N_{k,j}=\infty,\;\lim_{j\to\infty}\lim_{k\to\infty}\alpha_{k,j}N_{k,j}=\infty,
\end{align}
and 
\begin{align}\label{limit}
\beta'\mapsto \lim_{j\to\infty}\lim_{k\to\infty} \int_0^1 \frac{t}{\alpha_{k,j} N_{k,j}}\E\log \frac{Z\big(H_{N_{k,j},\xi^t}+H'_{N_{k,j},\alpha_{k,j},\beta'}\big)}{Z(H_{N_{k,j},\xi^t})}\d t\quad\mbox{exists everywhere,}
\end{align}
where the disorder $\xi^t$ is defined by (\ref{def:inter_disorder}).
Then for all $n\in \Bbb N$, it holds that
\begin{align}\label{main2}
\lim_{\ms D\ni \beta'\to  0}\lim_{j\to\infty}\varlimsup_{k\to\infty}\sup_{F:\Sigma^n_{N_{k,j}}\to [-1,1]}\left|\E\langle F\rangle_{H_{N_{k,j},\xi}+ H'_{N_{k,j},\alpha_{k,j},\beta'}}-\E\langle F\rangle_{H_{N_{k,j},g}+H'_{N_{k,j},\alpha_{k,j},\beta'}}\right|=0,
\end{align}
where $\ms D$ is the set of $\beta'\in \R$ at which the convex function defined by (\ref{limit}) is differentiable.
\end{thm}

Let us explain (\ref{aN-hyp}). First, since $\E X_{\alpha N}=\alpha N$ and $\E J_i=0$, the proof of \cite[Theorem~12.2.1]{Talagrand_v2} shows that
\begin{align}\label{def:qN}
q_{N,\xi,\alpha}(\beta')\doteq \frac{1}{\alpha N}\E\log \frac{Z(H_{N,\xi}+H'_{N,\alpha,\beta'})}{Z(H_{N,\xi})}
\end{align}
satisfies 
\begin{align}\label{ineq:qN}
0\leq q_{N,\xi,\alpha}(\beta')\leq \beta'.
\end{align}
Hence, if $\alpha_{N,M}\in (0,\infty)$ are constants such that the following analogues of the first and third assumed limits in (\ref{aN-hyp}) hold:
\begin{align}\label{aN-hyp-1}
\lim_{N\to\infty}\alpha_{N,M}=0\quad\mbox{and }\lim_{M\to\infty}\lim_{N\to\infty}\alpha_{N,M}N=\infty,
\end{align} 
then the convex functions
\[
\beta' \lmt\int_0^1 \frac{t}{\alpha_{N,M}N}\E\log \frac{Z\big(H_{N,\xi^t}+H'_{N,\alpha_{N,M},\beta'}\big)}{Z(H_{N,\xi^t})}\d t;\quad N,M\in \Bbb N,
\]
are uniformly bounded on compacts by (\ref{ineq:qN}). Examples of (\ref{aN-hyp-1}) include the following two cases: 
\begin{enumerate}
\item [(1)] $\displaystyle \lim_{N\to\infty}\alpha_{N,M}N=\alpha_M\in (0,\infty)$ and $\displaystyle\lim_{M\to\infty}\alpha_M=\infty$;
\item [(2)] $\displaystyle\lim_{N\to\infty}\alpha_{N}N=\infty$ and $\displaystyle\lim_{N\to\infty}\alpha_N=0$.
\end{enumerate}
We impose
the somewhat tedious assumption (\ref{limit}), and hence, the limiting scheme in (\ref{main2}), since we do not know whether the standard subadditivity argument applies in this case. Nonetheless, by a standard result of convexity (cf. \cite[Theorem~10.8]{Rockfeller}),
one can choose subsequences $\alpha_{k,j}$ and $N_{k,j}$  from $\alpha_{N,M}$ and $N$  so that (\ref{limit}) is satisfied. 

\begin{proof}[\bf Proof of Theorem~\ref{thm:3}]
We write $\langle \,\cdot\,\rangle_{t,\alpha,\beta'}$ for $\langle \,\cdot\,\rangle_{H_{\xi^t}+H'_{\alpha,\beta'}}$ in this proof.

Plainly, Step~\hyperlink{step3-1}{1}--Step~\hyperlink{step3-2}{2}
in the proof of Proposition~\ref{prop:1} still apply even if we add the present perturbative Viana--Bray Hamiltonians to the interpolating Hamiltonians considered there.
Then an analogue of (\ref{univ}) for the present setup follows, and it is enough to consider the case $\beta_2\neq 0$ and show, for any fixed $n\in \Bbb N$,
\begin{align}\label{VB-4op}
\lim_{\ms D\ni \beta'\to  0}\lim_{j\to\infty}\varlimsup_{k\to\infty}\sup_{F:\Sigma^n_{N_{k,j}}\to [-1,1]}\left|\int_0^1 t\Bigg(\frac{1}{N^2}\sum_{\mathbf i\in \{1,\cdots,N\}^2}\E\langle \D_{\mathbf i}^4F\rangle_{t,\alpha_{k,j},\beta'}\Bigg)\d t\right|=0.
\end{align}
The plan below for (\ref{VB-4op}) may be compared to the derivation of the Ghirlanda--Guerra identities. 
 See \cite[Section~4]{De_Sanctis_2009}.  \\

\noindent \hypertarget{step4-1}{{\bf Step 1.}}
We show that \begin{align}\label{SA:VB}
\begin{split}
 \lim_{j\to\infty}\varlimsup_{k\to\infty}\sup_{F:\Sigma^n_{N_{k,j}}\to [-1,1]}\int_0^1t\Bigg|\E\left\langle  \frac{H'_{N_{k,j},\alpha_{k,j},\beta'}(\spin^1)}{\alpha_{k,j} N_{k,j}\beta'}\Delta_1 F\right\rangle_{t,\alpha_{k,j},\beta'}\Bigg|\d t=0,\;\;\beta'\in \ms D\setminus\{0\},
 \end{split}
\end{align}
where $\ms D$ is defined in the statement of Theorem~\ref{thm:3} and $\Delta_1 F$ is defined by (\ref{def:DeltaF}).

We recall an inequality due to Panchenko which, adapted to the present setup, takes the following form when applied to the test functions $\Delta_1F$: for any $\beta''>\beta'\neq 0$,
\begin{align}
\sup_{F:\Sigma^n_N\to [-1,1]}\Bigg|\E\left\langle\frac{H'_{N,\alpha,\beta'}(\spin^1)}{\alpha N\beta'} 
\Delta_1 F\right\rangle_{t,\alpha,\beta'}\Bigg|\leq 
2\sqrt{\frac{\vep^N_{t,\alpha}(\beta',\beta'')}{\alpha N(\beta''-\beta')}}+8\vep^N_{t,\alpha}(\beta',\beta'')\label{Peq}
\end{align}
(see the proof of \cite[(3.102) in Theorem~3.8]{Panchenko_2013}).
Here, we set
\begin{align*}
\vep^N_{t,\alpha}(\beta',\beta'')\doteq &\frac{1}{\alpha  N}\int_{\beta'}^{\beta''}\E\left\langle \Bigg(\sum_{i=1}^{ X_{\alpha N}}J_i\sigma_{\mathbf u_i,\mathbf v_i}-\left\langle \sum_{i=1}^{ X_{\alpha N}}J_i\sigma_{\mathbf u_i,\mathbf v_i}\right\rangle_{t,\alpha,x}\Bigg)^2\right\rangle_{t,\alpha,x} \d x\\
\leq &\frac{q_{N,\xi^t,\alpha}(\beta''+y)-q_{N,\xi^t,\alpha}(\beta'')}{y}-\frac{q_{N,\xi^t,\alpha}(\beta')-q_{N,\xi^t,\alpha}(\beta'-y)}{y},\quad \forall\; y\in (0,\infty),
\end{align*}
and the function $q_{N,\xi^t,\alpha}(\beta')$ is defined by (\ref{def:qN}). Since $\E  X_{\alpha N}=\alpha N$, it follows from the mean-value theorem that the foregoing bound for $\vep^N_{t,\alpha}(\beta',\beta'')$ is uniformly bounded by $2$.  

The required limit (\ref{SA:VB}) thus follows from (\ref{Peq}) and the third assumed limit in (\ref{aN-hyp}) if we pass the following limits in order: $k\to\infty$, $j\to\infty$, $\beta''\to \beta'$ and finally $y\to 0$ (cf. the proof of \cite[Theorem~3.8]{Panchenko_2013}). 
\\

\noindent \hypertarget{step4-2}{{\bf Step 2.}} By the Poisson integration by parts, we get, for $\beta'\neq 0$ and $F:\Sigma^n_N\to [-1,1]$,
\begin{align}\label{PIBP}
\E\left\langle \frac{H'_{N_{k,j},\alpha_{k,j},\beta'}(\spin^1)}{\alpha_{k,j} N_{k,j}\beta'}\Delta_1 F\right\rangle_{t,\alpha_{k,j},\beta'}=\E \left\langle \frac{ J\prod_{i=1}^{n+1}\big(1+ \tanh(\beta'J)\sigma^i_{\mathbf u,\mathbf v}\big)\sigma^1_{\mathbf u,\mathbf v}\Delta_1 F}{\big(1+ \tanh(\beta'J)\langle \sigma_{\mathbf u,\mathbf v}\rangle_{t,\alpha_{k,j},\beta'}\big)^{n+1}}\right\rangle_{t,\alpha_{k,j},\beta'},
\end{align} 
where $J,\mathbf u,\mathbf v$ are independent of the Hamiltonian defining $\langle \,\cdot\,\rangle_{t,\alpha_{k,j},\beta'}$ and are distributed as $J_1$, $\mathbf u_1,\mathbf v_1$, respectively (cf. \cite{Guerra_2004, De_Sanctis_2009}). Consider the Taylor series of the $\E\langle \,\cdot\,\rangle_{t,\alpha_{k,j},\beta'}$-integrand on the right-hand side of (\ref{PIBP}) in $\tanh(\beta'J)$. We study explicit forms of the Taylor coefficients in this step and then pass suitable limits for the terms 
in the remaining steps. The reader will see particular forms of the Ghirlanda--Guerra identities  in the first-order coefficients and the normalized sums in the required equality (\ref{VB-4op}) in the third-order coefficients.

To obtain explicit forms of the Taylor coefficients, we use the following  identity:
\begin{align}
\frac{\sum_{m=0}^{n+1}A_mx^m}{(1-Bx)^{n+1}}
=&\sum_{m=0}^\infty \left(\sum_{a=0}^{\min\{m,n+1\}}{n+m-a\choose n }A_aB^{m-a}\right)x^m, \label{Taylor1}
\end{align}
where the binomial coefficients follow by counting the number of solutions $(m_1,\cdots,m_{n+1})\in \Bbb Z_+^{n+1}$ to $m_1+m_2+\cdots+m_{n+1}=m$ for every $m\in \Bbb Z_+$ (cf. \cite[page 18]{Stanley_book}). The series on the right-hand side of (\ref{Taylor1}) converges absolutely for $|x|<1/|B|$.
By (\ref{Taylor1}), the $\E\langle \,\cdot\,\rangle_{t,\alpha_{k,j},\beta'}$-expectation of 
the $m$-th term, $m\in\Bbb Z_+$, in the Taylor series under consideration is given by the product of $\E J\tanh^m(\beta' J)$ and the following term:
\begin{align}
\begin{split}\label{def:Aj}
A_m(t,\alpha_{k,j},\beta')
\doteq &\sum_{a=0}^{\min\{m,n+1\}}(-1)^{m-a}{n+m-a\choose n}\\
&\times\sum_{1\leq \ell_1<\ell_2<\cdots<\ell_a\leq n+1}\E\langle \sigma^{\ell_1,\cdots,\ell_a,1}_{\mathbf u,\mathbf v}\Delta_1 F\rangle_{t,\alpha_{k,j},\beta'} \langle \sigma_{\mathbf u, \mathbf v}\rangle_{t,\alpha_{k,j},\beta'}^{m-a}
\end{split}\\
=&\frac{1}{m!}\E\langle S^{m,n+1}_{\mathbf u,\mathbf v}\sigma^1_{\mathbf u,\mathbf v}\Delta_1 F\rangle_{t,\alpha_{k,j},\beta'}\label{Aj:overlap},
\end{align}
where the notation (\ref{def:Sj}) is used in (\ref{Aj:overlap}).
With the convention as in (\ref{convention:Dj}), we can replace $\min\{m,n+1\}$ with $m$ in the definition (\ref{def:Aj}) of $A_m(t,\alpha_{k,j},\beta')$, and the sum is not changed. Notice that $A_m(t,\alpha_{k,j},\beta')$ is bounded by a constant depending only on $m$ and $n$, and  we can write
\begin{align}\label{Ajoverlap}
\E\langle \sigma^{\ell_1,\cdots,\ell_a,1}_{\mathbf u,\mathbf v}\Delta_1 F\rangle_{t,\alpha_{k,j},\beta'} \langle \sigma_{\mathbf u, \mathbf v}\rangle_{t,\alpha_{k,j},\beta'}^{m-a}=\E\langle R^2_{\ell_1,\cdots,\ell_a,1,n+2,\cdots,n+1+m-a}\Delta_1 F\rangle_{t,\alpha_{k,j},\beta'}
\end{align}
by (\ref{sum:RR}) 
so that the multi-overlaps defined by (\ref{def:4overlap}) come into play through $A_m(t,\alpha_{k,j},\beta')$'s according to (\ref{def:Aj}). \\

\noindent \hypertarget{step4-3}{{\bf Step 3.}} 
Now by the discussion in Step~\hyperlink{step4-2}{2}, we get, for $\beta'\neq 0$,
\begin{align}
&\E\left\langle \frac{H'_{N_{k,j},\alpha_{k,j},\beta'}(\spin^1)}{\alpha_{k,j} N_{k,j}\beta'}\Delta_1 F\right\rangle_{t,\alpha_{k,j},\beta'}\notag\\
\begin{split}\label{SA:ALG0}
=&\,\sum_{m=1}^3\E J\tanh^m(\beta' J)\cdot A_m(t,\alpha_{k,j},\beta')+\mathcal O\big((\beta')^4\big)
\end{split}\\
=&\,\E J\tanh(\beta' J)\cdot A_1(t,\alpha_{k,j},\beta')+\E J\tanh^3(\beta'J)\cdot A_3(t,\alpha_{k,j},\beta')+o\big((\beta')^{3}\big),\quad \beta'\to 0,\label{SA:ALG}
\end{align}
where (\ref{SA:ALG0}) starts with $\E J\tanh(\beta' J)$ since $\E J=0$ and the term $\mathcal O\big((\beta')^4\big)$ in (\ref{SA:ALG0}) follows since $J$ is bounded.
Also, we get the term $o\big((\beta')^3\big)$ in (\ref{SA:ALG}) because the summand in (\ref{SA:ALG0}) indexed by $m=2$ can be handled as follows:
\[
\lim_{\beta'\to 0}\E J\frac{ \tanh^2(\beta'J)}{(\beta')^3}=
\lim_{\beta'\to 0}\E J\left(\frac{ \tanh^2(\beta'J)-(\beta' J)^2}{(\beta')^3}\right)=\E\lim_{\beta' \to 0}J\left(\frac{ \tanh^2(\beta'J)-(\beta' J)^2}{(\beta')^3}\right)=0.
\]
Here, the first equality uses the assumption that $\E J^3=0$ and
 the last equality follows from the Taylor series of the even function $x\mapsto \tanh^2(x)$, the boundedness of $J$ and dominated convergence.

Notice that the limiting behaviors of the terms $\mathcal O\big((\beta')^4\big)$ and $o\big((\beta')^{3}\big)$ in (\ref{SA:ALG0}) and (\ref{SA:ALG}) are uniform in $k,j\in \Bbb N$, $t\in [0,1]$ and $F:\Sigma^n_{N_{k,j}}\to [-1,1]$.\\

\noindent \hypertarget{step4-4}{{\bf Step 4.}} 
In this step, we show that 
\begin{align}
\mbox{}\lim_{\ms D\ni \beta'\to 0} \lim_{j\to\infty}\varlimsup_{k\to\infty}\sup_{F:\Sigma^n_{N_{k,j}}\to [-1,1]}\left|\int_0^1 t A_3(t,\alpha_{k,j},\beta') \d t\right|=0.  \label{lim:A3}
\end{align}

By dominated convergence, the elementary inequality $|\tanh(x)|\leq |x|$ and $\E J^4\in (0,\infty)$, it holds that
\[
\lim_{\beta'\to 0}\frac{o\big((\beta')^3\big)}{\E J\tanh^3 (\beta' J) }=0.
\]
Hence, we obtain from (\ref{SA:ALG}) that
\begin{align}
&\lim_{\ms D\ni \beta'\to 0} \lim_{j\to\infty}\varlimsup_{k\to\infty}\sup_{F:\Sigma^n_{N_{k,j}}\to [-1,1]}\left|\int_0^1 t A_3(t,\alpha_{k,j},\beta')\d t\right|\notag \\
\leq &\lim_{\ms D\ni \beta'\to 0} \lim_{j\to\infty}\varlimsup_{k\to\infty}\sup_{F:\Sigma^n_{N_{k,j}}\to [-1,1]}\left|\frac{1}{\E J\tanh^3(\beta' J)}\int_0^1t\E\left\langle \frac{H'_{N_{k,j},\alpha_{k,j},\beta'}(\spin^1)}{\alpha_{k,j} N_{k,j}\beta'}\Delta_1 F\right\rangle_{t,\alpha_{k,j},\beta'}\d t\right|\notag\\
&+\lim_{\ms D\ni \beta'\to 0} \lim_{j\to\infty}\varlimsup_{k\to\infty}\sup_{F:\Sigma^n_{N_{k,j}}\to [-1,1]}\left|\frac{\E J\tanh(\beta' J)}{\E J\tanh^3(\beta' J)}\int_0^1 t A_1(t,\alpha_{k,j},\beta')\d t\right|\notag\\
=&\lim_{\ms D\ni \beta'\to 0} \lim_{j\to\infty}\varlimsup_{k\to\infty}\sup_{F:\Sigma^n_{N_{k,j}}\to [-1,1]}\left|\frac{\E J\tanh(\beta' J)}{\E J\tanh^3(\beta' J)}\int_0^1 t A_1(t,\alpha_{k,j},\beta')\d t\right|,\label{SA:VB2}
\end{align}
where the last equality follows from (\ref{SA:VB}).

By the proof of \cite[Theorem~12.2.1]{Talagrand_v2} and the first and second assumed limits in (\ref{aN-hyp}), we have 
\begin{align}\label{fe_notchanged}
\lim_{k\to\infty} \frac{1}{ N_{k,j}}\E\log \frac{Z\big(H_{N_{k,j},\xi^t}+H'_{N_{k,j},\alpha_{k,j},\beta'}\big)}{ Z\big(H_{N_{k,j},\xi^t}\big)}=0,\quad\forall\; j,\,t,\,\beta',
\end{align}
so that Theorem~\ref{thm:1} and Corollary~\ref{cor:1} also apply under the Gibbs expectations defined by $H_{N_{k,j},\xi^t}+H'_{N_{k,j},\alpha_{k,j},\beta'}$ in the limit $k\to\infty$. 
Hence, by the explicit form (\ref{def:Aj}) of $A_1(t,\alpha_{k,j},\beta')$ and the assumption $\beta_2\neq 0$, we have
\begin{align}\label{lim:A1=0}
\lim_{k\to\infty}\sup_{F:\Sigma^n_{N_{k,j}}\to [-1,1]}|A_1(t,\alpha_{k,j},\beta')|
 =0,\quad \forall\;j,t,\beta'.
\end{align}
The required equality (\ref{lim:A3}) thus follows from the foregoing equality and (\ref{SA:VB2}).\\

\noindent \hypertarget{step4-5}{{\bf Step 5.}} 
We complete the proof in this step by showing that (\ref{lim:A3}) implies  (\ref{VB-4op}).

Let $\langle\,\cdot\,\rangle$ be the Gibbs expectation where the Hamiltonian is the sum of a general Hamiltonian over $N$ sites and a mixed $p$-spin Hamiltonian as in (\ref{def:H}). Observe that
\begin{align}\label{Dialt}
\langle \D_\i F\rangle=\sum_{\ell=1}^n \langle \sigma^\ell_\i\Delta_\ell F\rangle.
\end{align}
Hence,
applying (\ref{Gibbs:der}) and the continuity of the Gibbs expectations in temperature parameters
in the first and last equalities below and Proposition~\ref{prop:Dj} and (\ref{eg:A}) in the second and third equalities, we obtain
\begin{align*}
\langle \D^4_\i F\rangle =\sum_{\ell=1}^n \langle \D^3_\i \sigma^\ell_\i\Delta_\ell F\rangle
=&\sum_{\ell=1}^n \langle S_\i^{3,n+1}\sigma^\ell_\i\Delta_\ell F\rangle-2\sum_{\ell=1}^n\langle S_\i^{1,n+1}\sigma^\ell_\i\Delta_\ell F\rangle\\
=& \sum_{\ell=1}^n \langle S_\i^{3,n+1}\sigma^\ell_\i\Delta_\ell F\rangle-2\sum_{\ell=1}^n \langle \D^1_\i \sigma^\ell_\i\Delta_\ell F\rangle\\
=&\sum_{\ell=1}^n \langle S_\i^{3,n+1}\sigma^\ell_\i\Delta_\ell F\rangle -2\langle \D^2_\i F\rangle.
\end{align*}
Then from the last equality, we get
\begin{align}
&\Bigg|\int_0^1 t\Bigg(\frac{1}{N^2}\sum_{\mathbf i\in \{1,\cdots,N\}^2}\E\langle \D_{\mathbf i}^4F\rangle_{t,\alpha_{k,j},\beta'}\Bigg)\d t\Bigg|\notag\\
\leq &\sum_{\ell=1}^n\Bigg|\int_0^1 t 
\Bigg(\E\langle S_{\mathbf u,\mathbf v}^{3,n+1}\sigma^\ell_{\mathbf u,\mathbf v}\Delta_\ell F\rangle_{t,\alpha_{k,j},\beta'}\Bigg) \d t\Bigg|+2\Bigg|\int_0^1 t\Bigg(\frac{1}{N^2}\sum_{\mathbf i\in \{1,\cdots,N\}^2}\E\langle \D_{\mathbf i}^2F\rangle_{t,\alpha_{k,j},\beta'}\Bigg)\d t\Bigg|.\label{D4final}
\end{align}
The first term in (\ref{D4final}) tends to zero  in the same mode of uniform convergence in (\ref{VB-4op}), if we apply (\ref{Aj:overlap}), (\ref{lim:A3}) and the analogues of (\ref{lim:A3}) where $\sigma^1_{\mathbf u,\mathbf v}\Delta_1 F$ are replaced by $\sigma^\ell_{\mathbf u,\mathbf v}\Delta_\ell F$ for $2\leq \ell\leq n$. As in the proof of (\ref{lim:A1=0}),
we have the same mode of uniform convergence to zero of the second term in (\ref{D4final}) by
Corollary~\ref{cor:1}, Proposition~\ref{prop:Dj} and (\ref{eg:A}). This proves
(\ref{VB-4op}). 
The proof is complete. 
\end{proof}

\section{Approximate integration by parts with cutoffs}\label{sec:AIBP}
In this section we prove the following result, which is used throughout this paper. See also \cite[Section~3]{C_05}.

\begin{prop}\label{prop:AIBP}
Let $\xi$ be a real-valued random variable with zero mean and unit variance.
Then for any $f:\R\to \R$ with a bounded continuous third-order derivative, we have
\begin{align}\label{IBP2}
\E\xi f(\xi)=\E \partial_xf(\xi)+\error^2_\xi\big(\partial_x^2f\big),
\end{align}
where 
\begin{align}\label{gamma2}
\error^2_\xi\big(\partial_x^2f\big)\doteq \E\xi\int_0^\xi (\xi-x)\partial_x^2f(x)\d x-\E\int_0^\xi(\xi-x)\partial_x^3f(x)\d x.
\end{align}
The first expectation on the right-hand side of (\ref{gamma2}) is well-defined since
\begin{align}
\label{IBP2-1}
\left|\xi\int_0^\xi (\xi-x)\partial_x^2f(x)\d x\right|\leq &\,|\xi|\int_0^{|\xi|}\min\big\{2\|\partial_xf\|_\infty,\|\partial_x^2f\|_\infty x\big\}\d x.
\end{align}
\end{prop}
\begin{proof}[\bf Proof.]
By Taylor's theorem, the following equation holds:
\begin{align}\label{Phi:taylor}
\begin{split}
\xi f(\xi)=&\,\xi f(0)+\xi^2 \partial_x f(0)+ \int_0^\xi \partial^2_x f(x)\d x\\
&+\int_0^\xi \left(-\partial_x^2f(\xi)+\int_x^\xi \partial_x^3 f(y)\d y\right)\d x+\xi\int_0^\xi (\xi-x)\partial^2_xf(x)\d x.
\end{split}
\end{align}
The assumption $\E\xi=0$ and $\E\xi^2=1$ implies that the sum of expectations of the first three terms on the right-hand side of (\ref{Phi:taylor}) reduces to $\E\partial_x f(\xi)$. Also, expectations of the last two terms in (\ref{Phi:taylor}) sum to the last term in (\ref{IBP2}) since 
\[
\int_0^\xi -\partial_x^2 f(\xi)\d x=-\xi\partial_x^2 f(0)-\xi\int_0^\xi\partial_x^3 f(x)\d x
\]
and $\E\xi=0$. We have proved (\ref{IBP2}). To see (\ref{IBP2-1}), one may
use the identity: 
\[
\xi\int_0^\xi (\xi-x)\partial_x^2 f(x)\d x=\xi \int_0^\xi \big(\partial_xf(x)-\partial_x f(0)\big)\d x
\]
and the mean-value theorem. \end{proof}

\end{document}